\documentclass{amsproc}

\usepackage{graphicx}
\usepackage{amsmath,amsthm,amscd,amssymb}
\usepackage{amsfonts}
\usepackage{latexsym}
\usepackage{epsfig,epstopdf}
\newtheorem{thm}{Theorem}
\numberwithin{thm}{section}
\newtheorem*{thmA*}{Theorem A}
\newtheorem*{thm*}{Theorem A*}
\newtheorem{lemma}[thm]{Lemma}
\newtheorem{prop}[thm]{Proposition}
\newtheorem{cor}[thm]{Corollary}

\theoremstyle{definition}
\newtheorem{definition}[thm]{Definition}

\theoremstyle{remark}
\newtheorem{rmk}[thm]{Remark}
\newtheorem{claim}[thm]{Claim}

\numberwithin{equation}{section}

\newcommand{\wt}{\widetilde{W}}
\newcommand{\te}{\tilde{e}}

\newcommand{\dH}{\dot{H}}

\newcommand{\NLS}{\mbox{\rm NLS}}
\newcommand{\GB}{\mbox{\rm GBG}}
\newcommand{\g}{\mathcal{G}}
\newcommand{\Pu}{\mathcal{P}}
\newcommand{\ME}{\mathcal{ME}}
\newcommand{\N}{\mathbb{N}}
\newcommand{\C}{\mathbb{C}}
\newcommand{\Rn}{\mathbb{R}^2}
\newcommand{\R}{\mathbb{R}}
\newcommand{\dHs}{\dot{H}^{1/2}}
\newcommand{\dHds}{\dot{H}^{-1/2}}
\newcommand{\NLSfp}{\NLS^+_5(\Rn)}
\newcommand{\Lt}{L^2}
\newcommand{\SHs}{S(\dHs)}
\newcommand{\SLt}{S(\Lt)}

\newcommand{\SLTx}{S(\Lt;[T,\infty))}
\newcommand{\dSLTx}{S^{\prime}(\Lt;[T,\infty))}
\newcommand{\SHSTx}{S(\dHs;[T,\infty))}

\newcommand{\uQ}{Q}
\newcommand{\NLSf}{\NLS^{+}_{ 5}(\Rn)}
\newcommand{\tpsi}{\widetilde{\psi}}
\newcommand{\tu}{\tilde{u}}
\newcommand{\at}{\mathcal{T}}
\newcommand{\SdHs}{S^{\prime}(\dHds)}
\newcommand{\Ds}{D^{1/2}}
\newcommand{\nonlinealu}{|u|^4u}
\newcommand{\nonlinealQ}{|Q|^4Q}
\newcommand{\crit}{{\textnormal{c}}}

\newcommand{\ds}{\displaystyle}
\DeclareMathOperator{\im}{Im}

\begin{document}

\title[Scattering and Blow up]{Scattering and Blow up for the Two Dimensional Focusing Quintic Nonlinear Schr\"odinger Equation}

\author{Cristi Guevara}
\address{School of Mathematical and Statistical Sciences, Arizona State University, Tempe, Arizona, 85287}
\curraddr{School of Mathematical and Statistical Sciences, Arizona State University, Tempe, Arizona, 85287}
\email{Cristi.guevara@asu.edu}
\thanks{C.G. was partially supported by grants from the National Science Foundation (NSF - Grant DMS - 0808081 and NSF - Grant DUE-0633033; PI Roudenko), the Alfred P. Sloan Foundation and would like to thank Gustavo Ponce for discussions on the subject and Svetlana Roudenko for guidance on this topic. }

\author{Fernando Carreon}
\address{Department of Mathematics,
University of Michigan, Ann Arbor, Michigan 48109}
\email{carreonf@umich.edu}
\thanks{F. C was partially supported by grants from the National Science Foundation (NSF - Grant DMPS-0838704), the National Security Agency (NSA - Grant H98230-09-1-0104), the Alfred P. Sloan Foundation and the Office of the Provost of Arizona State University.}



\begin{abstract}
Using  the concentration-compactness method  and the localized virial type arguments, we study the behavior of $H^1$ solutions to the focusing quintic NLS in $\Rn$, namely, 
$$i \partial_t u+\Delta u+\nonlinealu=0,\quad\quad (x, t) \in \Rn\times\R.$$

 Denoting by $M[u]$ and $E[u]$, the mass and energy of a solution $u,$ respectively, and $Q$ the ground state solution 
to $-Q+\Delta Q+ \nonlinealQ=0$, and assuming $M[u]E[u] <M[Q]E[Q]$, we characterize the threshold for global versus finite time existence. Moreover, we show scattering for global existing time solutions and finite or ``weak" blow up for the complement region. 
This work is in the spirit of \cite{KeMe06} and \cite{DuHoRo08, HoRo08, HoRo09}.
\end{abstract}

\maketitle

\section{Introduction}

Consider the  focusing quintic nonlinear Schr\"odinger equation on $\Rn$
\begin{align}
	\left\{\begin{array}{ccc}i \partial _t u+\Delta u+|u|^4u=0  \\u(x,0)=u_0(x)\in H^1(\Rn), \end{array}\right.
	\label{quintic}
\end{align}
where $u=u(x,t)$ is a complex-valued function in space-time $\Rn_x\times \R_t$.   

The initial-value problem \eqref{quintic} is locally well-posed in $H^1$ (see Ginibre-Velo \cite{GiVe79a}). Let  $I=(-T_*,T^*)$ be the maximal interval of existence in time of solutions to \eqref{quintic}.    Solutions  to \eqref{quintic} on $(-T_*,T^*)$  satisfy mass conservation $M[u](t)=M[u_0]$, energy conservation $E[u](t)=E[u_0]$ and momentum conservation $P[u](t)=P[u_0]$, where
\begin{align*}
	M[u](t)=\int_{\Rn} |u(x,t)|^2dx, 
\end{align*}
\begin{align*}
	E[u](t)=\dfrac 12 \int_{\Rn}|\nabla u(x,t)|^2dx-\dfrac 1{6} \int_{\Rn}|u(x,t)|^{6} dx,
\end{align*}
\begin{align*}
	P[u](t)=Im \int_{\Rn}\bar{u}(x,t)\nabla u(x,t)dx.
\end{align*}
The NLS equation has several symmetries and for the purpose of this paper we discuss two of them. If $u( x,t)$ is a solution to \eqref{quintic}, the Galilean invariant  $u_G$
 \begin{align}
 u_G( x,t)=e^{ix\cdot\xi_0}e^{-it|\xi_0|^2}u(x-(x_0+2\xi_0t),t)\label{ugalilean}
 \end{align}
 also is a solution.
  
Observe that for a fixed $\lambda \in (0,\infty)$, if $u(x, t)$ solves \eqref{quintic}, then  $u_{\lambda}(x,t):=\lambda^{\frac12}u(\lambda x, \lambda^2 t)$ solves \eqref{quintic}. This scaling preserves the $\dHs(\Rn)$ norm, thus, the initial value problem  \eqref{quintic} is known as an  $\dHs$-{\it critical} problem, hence,  it is mass-supercritical and energy-subcritical. The purpose of this paper is to investigate global behavior of solutions (in time) for the  Cauchy problem \eqref{quintic} with $u_0\in H^1(\Rn)$.

Small data theory guarantees the global existence and scattering for solutions to \eqref{quintic} with initial condition $\|u_0\|_{\dH^s(\Rn)}<\delta$ for small $\delta>0$ and $s\geq 1/2$. On the other hand, existence of blow up solutions is known from 1970's (see Vlasov-Petrishchev-Talanov \cite{VlPeTa71}, Zakharov \cite{Za72}, Glassey \cite{Gl77}) by convexity argument on variance $V(t)=\int |x|^2|u(x,t)|^2dx$ for solutions with negative energy $E[u]<0$ and finite variance $(V(0)<\infty)$.

We briefly review recent developments for global solutions to a general NLS 
\begin{align}
	\left\{\begin{array}{ccc}i \partial _t u+\Delta u+|u|^{p-1}u=0  \\u(x,0)=u_0(x). 
	\end{array}\right.
	\label{general}
\end{align}

For studying long-term behavior of solutions in the energy-critical focusing  case of NLS \eqref{general} (for  $p=\frac4{n-2}+1$,   $u_0\in \dH^1(\R^n)$, and $n=3, 4, 5$), Kenig-Merle \cite{KeMe06} applied the concentration-compactness and rigidity technique. The concentration-compactness method appears first in the context of wave equation in G\'erard \cite{Ge96} and  NLS in Merle-Vega \cite{MeVe98}, which was later followed by  Keranni \cite{Ke01}, and dates back to P.L.  Lions \cite{Lions84} and Brezis-Coron \cite{BrCo85}. The localized variance estimates are  due to F. Merle from mid 1980's. In \cite{KeMe06} the authors obtain a sharp threshold for scattering and finite time blow up for radial initial data for solutions with $E[u]<E[W]$. 

In  the case of the 3d focusing cubic NLS (a  mass-supercritical and energy-subcritical problem)  equation with $H^1$ initial data this method was applied to obtain scattering for global existing solutions under the mass-energy threshold (i.e., $M[u]E[u]<M[Q]E[Q]\;$) by  Holmer-Roudenko for radial functions in \cite{HoRo08}, Duyckaerts-Holmer-Roudenko for nonradial functions in \cite{DuHoRo08}.   Duyckaerts-Roudenko  in \cite{DuRo08} obtain the characterization of all solutions at the threshold $M[u] E[u]=M[Q]E[Q]$. Furthermore, for infinite variance nonradial solutions Holmer-Roudenko \cite{HoRo09} established a version of the blow up result (in this paper refereed as ``weak"  blow up), meaning that either blow up occurs in finite time   ($T^*<+\infty$), or $T^*=+\infty$ and there exists a time sequence $\{t_n\}\to+\infty$ such that $\|\nabla u(t_n)\|_{\Lt}\to+\infty.$ This last result is the first application of the concentration compactness and rigidity arguments to establish the divergence property of solutions as opposed to scattering when these techniques are used to show some boundedness properties of solutions.

In the spirit of  \cite{DuHoRo08,HoRo08, HoRo09}  we analyze the global behavior of solutions for the focusing quintic NLS in two dimensions \eqref{quintic}, denoted by $\NLSfp$.

Note that $u(x,t)=e^{it}Q(x)$ solves the equation \eqref{quintic},  provided $Q$ solves  
\begin{align}
-Q+\Delta Q+|Q|^4Q=0, \qquad Q=Q(x), \qquad x\in{\Rn}. \label{ground}
\end{align}

From the theory of  nonlinear elliptic equations denoted by Berestycki-Lions \cite{BeLi83a,BeLi83b}, it is known that the equation \eqref{ground} has  infinite number of solutions in $H^1(\Rn)$, but a unique solution of the minimal $\Lt$-norm, which we denote again by $Q(x)$. It is positive, radial, exponentially decaying (see \cite[Appendix B] {Tao06NLD}) and is called the {\it ground state} solution. 

Before stating our main result, we introduce the following notation:
 \begin{align*}
 &\bullet\text{the~renormalized~gradient} 
&\g_{u}(t)&:=\dfrac{\|u\|_{\Lt(\Rn)}\|\nabla u(t)\|_{\Lt(\Rn)}}{\|\uQ\|_{\Lt(\Rn)}\|\nabla \uQ\|_{\Lt(\Rn)}},\\
 &\bullet\text{the~renormalized~momentum} 
 &
 \Pu[u]&:=\dfrac {P[u]\|u\|_{\Lt(\Rn)}}{\|\uQ\|_{\Lt(\Rn)}\|\nabla \uQ\|_{\Lt(\Rn)}},
 \\
 &\bullet\text{the~renormalized~Mass-Energy}
 &
\ME[u]&:=\dfrac{M[u]E[u]}{M[\uQ]E[\uQ]}.
\end{align*}
\begin{rmk}[Negative energy] \label{d:negative energy}
Note that  it is possible to have initial data with $E[u]<0$ and the blowup from the dichotomy in Theorem A  Part II (a)  below applies. (It
follows from the standard convexity blow up argument and the work of
Glangetas-Merle \cite{GlMe}). Therefore, we only consider $E[u]\geq0$ in the rest of the paper.
\end{rmk}
The main result of this paper is the following
 \begin{thmA*}\label{main}
Let $u_0\in H^{1}(\Rn)$ and $u(t)$ be the corresponding solution to \eqref{quintic} in $H^1(\Rn)$ with maximal time interval of existence $(-T_*,T^*)$. 
Assume 
\begin{align}\label{mass-energy}
\ME[u]-2\Pu^2[u]<1.
\end{align} 
\begin{enumerate}
\item[I.] If  
\begin{align}\label{ground-momentum1}
\g^2_u(0)-\Pu^2[u]<1,
\end{align}
then
\begin{enumerate}
\item $\g^2_u(t)-\Pu^2[u]<1$ for all $t\in \R$, and hence, the solution is global in time  (i.e., $T_*,\;T^*=+\infty$), moreover,
\item $u$ scatters in $H^1(\Rn)$, this means, there exists $\phi_\pm\in H^1(\Rn)$ such that 
\begin{align*}\lim_{t\to \pm\infty}\|u(t)-e^{it\Delta}\phi_\pm\|_{H^1(\Rn)}=0.\end{align*}
\end{enumerate}
\item[II.] If   
\begin{align}\label{ground-momentum2}
\g^2_u(0)-\Pu^2[u]>1,
\end{align}
then  $\g^2_u(t)-\Pu^2[u]>1$ for all $t\in(-T_*,T^*)$ and if 
\begin{enumerate}
\item  $u_0$ is radial or  $u_0$ is of finite variance, i.e., $|x|u_0 \in \Lt(\Rn)$,  then the solution blows up in finite time in both time directions.
\item If $u_0$ non-radial and of infinite variance, then in the positive time direction either the solution blows up in finite time (i.e., $T^* < +\infty $) or there exists a sequence of times $t_n\to +\infty$ such that $\|\nabla u(t_n)\|_{\Lt(\Rn)}\to \infty$. Similar statement holds for $t<0$.
\end{enumerate}
\end{enumerate}
\end{thmA*}
To prove this theorem, we first reduce it to the solutions with zero momentum.  This is possible by Galilean transformation (see Section \ref{momentum}), and thus, we only prove  a reduced version of Theorem A, see the statement of  Theorem A* in Section \ref{momentum}.

Our arguments follow \cite{DuHoRo08, HoRo07, HoRo08,HoRo09} which considered the focusing  $\NLS_3(\R^3)$, however, several  non-trivial modifications had to be made. In particular,   
\begin{itemize}
\item The range of the Strichartz exponents is adapted for the two dimensional case, as well as the range of admissible pairs for the Kato-type estimate \eqref{Kato-Strichartz}, see Section \ref{local-theory} and also Remarks \ref{remark 24} and \ref{remark 26}.
\item  The pair $(2,\infty)$ is not $\dHs(\Rn)$-admissible (as oppose to $\R^3$ as was used in \cite{HoRo08}), thus,  when using Strichartz  and Kato estimates, we have to avoid this end point pair. To do that we use various interpolation tricks on other admissible pairs $(p,r)$ with $r<+\infty$, see Propositions \ref{small data}, \ref{H^1 Scattering} and \ref{longperturbation}.
\item We also note that there is a minor error in \cite[Proposition 2.2]{HoRo08} which we resolve in this paper,  see also errata \cite{HRerrata}. Refer to Remarks \ref{remark 24}, \ref{remark 26} and \ref{remark 28} discussing this matter.
\item The ground state, its variational characterization and Pohozhaev identities are different for the $\NLS_5(\Rn)$ (see Subsections \ref{sec:prop GS} and \ref{sec:var GS}).  
\item A new argument to obtain blow up for the radial data when $p=5$ (Theorem A II part (a)) was obtained. The approach in \cite{HoRo07} had a technical restriction, i.e., for $n\geq2$   the nonlinearity  $1+\frac4n<p<\min\{5,1+\frac4{n-2}\}$, and thus, would not include the case $p=5$. Combining estimates on the $L^6(\Rn)$ norm, the Gagliardo-Nierenberg estimate from \cite{OgTs91}  for radial functions  and the conservation of the mass, we resolve this issue. (However, for $n=2$, showing blow
up for $p > 5$ for radial data is still open.)
\item We explicitly state  the linear and the nonlinear profile decompositions in Section \ref{profile section} and ``general" existence of wave operator (Proposition \ref{Existence of wave operator}). General means in the sense that it can be applied later in both scattering and weak blow up parts of Theorem A. The  nonlinear profile decomposition for the 3d cubic NLS is hidden in \cite[Propositions 2.1 and 6.1]{DuHoRo08} as well in \cite{KeMe06}.
\end{itemize}

The structure of this paper is as follows: Section 2 reviews the local theory,  the properties of the ground state and  reduction of the problem with nonzero momentum  to the case $P[u]=0$ via Galilean transformation for the equation \eqref{quintic}. Section 3 states the blow up and scattering dichotomy results and existence of the wave operator for $\NLSfp$. In Section 4 we present the detailed proofs for the linear and nonlinear profile decompositions, these are the keys  of the technique. And finally, in Sections 5 - 6, we prove Theorem A, both based on the concentration compactness machinery and localized virial identity, in particular, in Section 5 we prove scattering  and in Section 6 we give the argument for  the ``weak" blow up (Theorem A II  (b)).

The arguments, presented in this paper, can be extended to other mass- supercritical and energy-subcritical $\NLS$ cases and we will  establish further generalizations elsewhere.

\subsection{Notation.} Through out the paper, most of the $L^p$, $H^s$ and $\dH^s$ norms are defined on $\Rn$, for example, $f\in L^p(\Rn)$ if $\|f\|_{L^p(\Rn)}^p=\int_{\Rn}|f(x)|^pdx<\infty.$  
In addition, we adopt  the notation $X  \lesssim Y$  whenever there exists some constant $c$, which does not depend on  the parameters, so that $X \leq c Y$. We 
denote $\NLS(t)\psi(x)$ the solution to \eqref{quintic}  with initial data $\psi(x)$.
 
 \section{Preliminaries}\label{SG}
\subsection{Local Theory} \label{local-theory} We first recall the  Strichartz estimates (e.g., see Cazenave \cite{Ca03}, Keel-Tao \cite{KeTa98}, Foschi \cite{Fos05}).

We say $(q,r)$ is $\dH^{s}-$ Strichartz admissible if 
\begin{align*}
\frac{2}{q}+\frac{2}{r}=1-s \qquad \mbox{ with } \quad 2 \leq q,r \leq \infty  \quad \mbox{ and }
  \quad(q,r) \neq (2,\infty).
\end{align*}

We will mainly consider $s=0$ ($\Lt$ admissible pairs) and $s=\frac12$ ($\dHs$ admissible pairs). 
Let 
\begin{align*}
	\| u \|_{\SLt} =  \sup_{\substack{(q,r)-
\Lt  \; {\rm admissible}
\\2^+\leq q \leq\infty,\quad
2\leq r \leq (2^+ )'}}
 \|u\|_{L^q_tL^r_x}\;. 
 \end{align*} 
Here, $ (a^+ )'$ is defined as $ (a^+ )' :=\frac{a^+\cdot a}{a^+-a},$  so that $ \frac1a=\frac1{(a^+ )'}+ \frac1{a^+ }$ for any positive real value $a$, with $a^+$ being a fixed number slightly larger than $a$.  Note that the choice of $ (a^+ )'$ guarantees that the sup is finite.  In particular, the pair $(2^+,(2^+ )')$ is still Strichartz admissible. Let
\begin{align*}
	\| u \|_{S'(\Lt)} =
	 \inf_{\substack{(q,r)-\Lt \; {\rm admissible}\\  
	 2^+\leq q \leq\infty^-,\quad
2\leq r \leq (2^+ )'}
	 }
	\| u \|_{L^{q'}_tL^{r'}_x}\;, 
 \end{align*} 
where $\infty^-$ stands for any large real number.
 Define the Strichartz norm $\SHs$ as
\begin{align*}
	\| u \|_{\SHs} = 
	\sup_{\substack{(q,r)-
\dHs \; {\rm admissible}
\\4^+\leq q \leq\infty,\quad
4^+ \leq r \leq(4^+ )'}}
 \|u\|_{L^q_t L^r_x}\;.
\end{align*}
Define the $S'(\dHds)$ norm
\begin{align*}
	\| u \|_{S'(\dHds)} = 
	\inf_{\substack{(q,r)-
\dH^{-\frac12} \; {\rm admissible}
\\\frac34^+\leq q \leq2^-,\quad
4^+ \leq r \leq(\frac34^+ )'}}
 \|u\|_{L^{q'}_t L^{r'}_x}\;,
\end{align*}
where $q'$ and $r'$ are the conjugates of $q$ and $r$, respectively.  In addition, the pair  $(2^-,4^+)$ is $\dH^{-\frac12}$ admissible.

The standard Strichartz estimates \cite{Ca03,KeTa98} are
 \begin{align}
\| e^{it \Delta} \phi \|_{\SLt} \leq c \| \phi \|_{L^{2} }
\quad\mbox{and}\quad
\Big\|\int_{0}^{t} e^{i(t-\tau) \Delta} f(\tau) d\tau \Big\|_{\SLt} \leq c \| f \|_{S'(\Lt)} \label{stri}.
\end{align} 
 By combining them with Sobolev embeddings yields
 \begin{align}
\| e^{it \Delta} \phi \|_{\SHs} \leq c \| \phi \|_{\dHs }
\quad\mbox{and}\quad
\Big\|\int^t_0 e^{i(t-\tau) \Delta} f(\tau) d\tau
\Big \|_{\SHs} \leq c \|\Ds  f \|_{S'(\Lt)}  \label{strisob}.
\end{align} 
Also recall the Kato-Strichartz estimate \cite{Fos05}
\begin{align}
\label{Kato-Strichartz}
\Big\|\int^t_0 e^{i(t-s) \Delta} f(\tau) d\tau\Big \|_{\SHs} \leq c \| f \|_{S'(\dHds)}.
\end{align} 

Note that the Kato-Strichartz estimate implies the second (inhomogeneous) estimate in \eqref{strisob} by Sobolev embedding but not vice versa. The Kato estimate is essential in the long term perturbation argument.

\begin{lemma}{\rm(Chain rule \cite{KePoVe93})} \label{chain}
Suppose $F \in C^1(\C)$ and $1<p,q,p_1,p_2,q_2 <\infty$, $1<q_1 \leq \infty$ such that
\begin{align*}
\frac{1}{p}=\frac{1}{p_1}+\frac{1}{p_2} \mbox{ \hspace{.2in} and \hspace{.2in}} \frac{1}{q}=\frac{1}{q_1}+\frac{1}{q_2}.
\end{align*}
Then
\begin{align}\label{eq:chain}
\| D^{1/2} F(f) \|_{L^p_x L^q_t} \leq c
 \| F'(f) \|_{L_x^{p_1} L_t^{q_1} } \| D^{1/2} f \|_{L^{p_2}_x L^{q_2}_t}\;.
\end{align}
\end{lemma}
\begin{lemma}{\rm(Leibniz rule \cite{KePoVe93})} \label{leibniz}
 Let  $1<p,p_1,p_2,p_3,p_4 <\infty$, 
 such that
\begin{align*}
\frac{1}{p}=\frac{1}{p_1}+\frac{1}{p_2} \quad\text{~~and~~}\quad\frac{1}{p}=\frac{1}{p_3}+\frac{1}{p_4} .
\end{align*} 
  Then
$\| D^{1/2} (fg) \|_{L^p } \lesssim \| f \|_{L^{p_1} } \| D^{1/2} g \|_{ L^{p_2}}+\| g\|_{L^{p_3} } \| D^{1/2} f \|_{ L^{p_4}}\;.
$
\end{lemma}

In what follows  we will use the $\Lt$--admissible pairs  $(6,3)$ and $(3,6)$; and the $\dHs$--admissible pairs  $(6,12)$ and $(8,8).$

\begin{prop}\label{small data}{\rm(Small data).}  
Suppose $\|u_0\|_{\dHs}\leq A.$ There exists $\delta_{sd}=\delta_{sd}(A)>0$ such that if $\|e^{it\Delta}u_0\|_{\SHs}\leq \delta_{sd}$, then u solving the $\NLS^+_5(\Rn)$ equation \eqref{quintic} is global in $\dHs$ and 
\begin{align*}
\|u\|_{\SHs}\leq 2 \|e^{it \Delta}u_0\|_{\SHs},\quad
\|\Ds u\|_{\SLt}\leq 2c \|u_0\|_{\dHs}\;.
\end{align*}
\end{prop}
\begin{proof}
Define the map $v\mapsto \Phi_{u_0}(u)$ via
$ \Phi_{u_0}(u)=e^{it\Delta}u_0+i \int_0^t e^{i(t-\tau)\Delta}|u|^4u(\tau)d\tau.
$ 
Let
\begin{align}\label{ball}
 B=\Big\{\|u\|_{\SHs}\leq2 \|e^{it \Delta}u_0\|_{\SHs},\quad \|\Ds u\|_{\SLt}\leq 2c \|u_0\|_{\dHs}\Big\}.
\end{align}
The argument is established by showing that $\Phi_{u_0}(u)$ is a contraction in the ball $B$.   By
triangle inequality and \eqref{strisob}, we have
\begin{align*}
\|\Phi_{u_0}(u)\|_{\SHs} &\leq
 \| e^{it\Delta} u_0\|_{\SHs} +\Big\|  \int_0^t e^{i(t-\tau)\Delta}|u|^4u(\tau)d\tau\Big\|_{\SHs}\\
&\leq \|e^{it\triangle}u_0\|_{\SHs} + c_1 \| \Ds |u|^4u\|_{S'(\Lt)},
\end{align*}
where $c_1$ takes care of the constants from \eqref{stri}. Applying the triangle inequality followed by \eqref{strisob} and since  $\|\Ds u_0\|_{\Lt}=  \|u_0\|_{\dHs}$, we obtain
\begin{align*}
\|\Ds \Phi_{u_0}(u)\|_{\SLt} 
\leq \|e^{it\Delta} \Ds  u_0\|_{\SLt} +\Big\|  \int_0^t e^{i(t-\tau)\Delta}\Ds|u|^4u(\tau)d\tau\Big\|_{\SLt} 
\\
\leq c_1\|\Ds u_0\|_{\Lt} +c_1\| \Ds (|u|^4u)\|_{S'(\Lt)} 
\leq c_1\|u_0\|_{\dHs} +c_1 \| \Ds |u|^4u\|_{S'(\Lt)}.
\end{align*} 

Then, we estimate the $S'(\Lt)$ norm by $L^{\frac{3}{2}}_tL^{\frac{6}{5}}_x$ norm (the pair (3,6) is an $L^2$ admissible), apply Chain rule Lemma \ref{chain} followed by the  H\"older's inequality, and finally, the $L^{8}_tL^{8}_x$   and $L^{6}_tL^{3}_x$ norms are estimated by the $\SHs$ norm and $\SLt$ norm, respectively: 
\begin{align*}
\| \Ds |u|^4u\|_{S'(\Lt)}&\leq \| \Ds |u|^4u\|_{L^{\frac{3}{2}}_tL^{\frac{6}{5}}_x}\leq c_2 \|u\|_{L^8_tL^8_x}^4 \|\Ds u\|_{L^6_tL^3_x} \\&\leq c_2 \|u\|_{\SHs}^4 \|\Ds u\|_{\SLt},
\end{align*}
where $c_2$ is the constant from \eqref{eq:chain}. Thus, the conditions in \eqref{ball} yield 
\begin{align}\label{small1}
\|\Phi_{u_0}(u)\|_{\SHs} 
&\leq \|e^{it\triangle}u_0\|_{\SHs} + c_1c_2 \|u\|_{\SHs}^4 \|\Ds u\|_{\SLt}\notag
\\
&\leq\left(1+32c_1c_2c\|e^{it\Delta}u_0\|^3_{\SHs}\|u_0\|_{\dHs}\right) \|e^{it\triangle}u_0\|_{\SHs} ,
\end{align}
and 
\begin{align}\label{small2}
\|\Ds \Phi_{u_0}(u)\|_{\SLt} &
\leq c_1\|u_0\|_{\dHs} +c_1 c_2 \|u\|_{\SHs}^4 \|\Ds u\|_{\SLt}\notag
\\
&
\leq c_1\|u_0\|_{\dHs}\left(1 +32 c_2c \|e^{it\Delta}u_0\|^4_{\SHs} \right).
\end{align} 
Thus, \eqref{small1} and \eqref{small2} imply
$$
32C\|e^{it \Delta}u_0\|^3_{\SHs} \|u_0\|_{\dHs}\leq 1\qquad \mbox{and}\qquad 32C\|e^{it \Delta}u_0\|^4_{\SHs} \leq 1
$$
and the contraction follows by letting $C=\max\{c_1,c_1c_2c,c_2c\}$ and choosing   $\delta_{sd} = \min\Big\{ \frac{1}{32CA^3}, \frac{1}{\sqrt[4]{32C}},\frac{1}{\sqrt[3]{32CA}}\Big\}.$
\end{proof}
{\rmk ({\it About the proof of Proposition }\ref{small data}) \label{remark 24}If we were to follow \cite[Proposition 2.1]{HoRo08} directly, in the inhomogeneous Strichartz estimates, we would write $\|\Ds v\|_{L^\infty_t\Lt_x},$ which would force us to estimate $\| \Ds |u|^4u\|_{L^{2}_tL^{1}_x}$. However, the pair $(2,1)$ is not an $\dHs$--admissible in ($\Rn$). To avoid this problem,  we choose the $\Lt$--admissible pair $(3,6)$ with its conjugate pair $(\frac32,\frac65)$, and estimate instead  $\|\Ds |u|^4u\|_{L^{\frac32}_tL^{\frac65}_x}$.}

\begin{prop}{\rm($H^1$ scattering).}\label{H^1 Scattering}
Assume $u_0 \in H^1$,  $u(t)$ is a global solution to (\ref{quintic}) with initial condition $u_0$, globally finite $\dHs$ Strichartz norm $\|u\|_{\SHs} <+\infty$ and uniformly bounded $H^1$ norm $\sup_{t\in[0,+\infty)}\|u(t)\|_{H^1}\leq B$. Then there exists $\phi_+\in H^1$ such that 
\begin{align}\label{h1scattering}
\lim_{t\to+\infty}\|u(t)-e^{it\Delta}\phi_+\|_{H^1}=0,
\end{align}
i.e., $u(t)$ scatters in $H^1$ as $t\to +\infty$. A similar statement holds for negative time.
\end{prop}
\begin{proof}
Since $u(t)$ solves (\ref{quintic}) with initial datum $u_0$,  we have  the integral equation
$\ds
u(t)=e^{it\Delta}u_0+i\int_0^t e^{i(t-\tau)\Delta}(|u|^4u)(\tau)d\tau.
$ 
 Define
\begin{align*}
\phi_+=u_0+i\int^{+\infty}_0 e^{-i\tau\Delta}(|u|^4u)(\tau)d\tau.
\end{align*}
Then
\begin{align}
u(t)-e^{it\Delta}\phi^+=-i\int^{+\infty}_t e^{i(t-\tau)\Delta}(|u|^4u)(\tau)d\tau \label{scatterH1}.
\end{align}
Estimating the $\Lt$ norm of \eqref{scatterH1} by Strichartz estimates and H\"older's inequality, we have 
\begin{align*}
\|u(t)-e^{it\Delta}\phi_+\|_{\Lt}
\lesssim \Big\|\int^{+\infty}_t e^{i(t-\tau)\Delta}(|u|^4u)(\tau)d\tau\Big\|_{\SLt}
\lesssim\| |u|^4u\|_{L^{\frac32}_{[t,\infty)}L^{\frac65}_x},
\end{align*}
and similarly, estimating the $\dH^1$ norm of  \eqref{scatterH1}, we obtain
\begin{align*}
\|\nabla(u(t)-e^{it\Delta}\phi_+)\|_{\Lt}
&\lesssim \Big\|\int^{+\infty}_t e^{i(t-\tau)\Delta}(\nabla(|u|^4u))(\tau)d\tau\Big\|_{\SLt}
\lesssim\| |u|^4\nabla u\|_{L^{\frac32}_{[t,\infty)}L^{\frac65}_x}\;.
\end{align*}
The Leibnitz rule yields
\begin{align*}
\|u(t)-e^{it\Delta}\phi_+\|_{H^1}&\lesssim \| |u|^4(1+\nabla)u\|_{L^{\frac32}_{[t,\infty)}L^{\frac65}_x}\\
&\lesssim\|u\|^4_{L^{6}_{[t,\infty)}L^{12}_x}\|(1+\nabla)u\|_{L^{\infty}_{[t,\infty)}L^{2}_x} 
\lesssim B\|u\|^4_{L^{6}_{[t,\infty)}L^{12}_x}.
\end{align*}
Note that the above estimate is obtained using the H\"older inequality with the  split $\frac23=\frac46+\frac1\infty$ and $\frac56=\frac4{12}+\frac12$, and the hypothesis $\sup_{t\in[0,+\infty)}\|u(t)\|_{H^1}\leq B$. And
as $t\to\infty$, $\|u\|_{L^{6}_{[t,\infty)}L^{12}_x} \to 0,$ thus  we obtain \eqref{h1scattering}.
\end{proof}
\begin{rmk}\label{remark 26}
The above proof is a direct application of the strategy from \cite[Proposition 2.2]{HoRo08}, 
namely, we find that  \eqref{scatterH1} is bounded in the $H^1$ norm by the Strichartz norm  $S(L^{2}({[t,\infty)},\Rn)$,  which diminishes to 0 as $t\to\infty$.  However, this procedure fails in the case of  $\NLS_3(\R^3)$ as written in \cite[Proposition 2.2]{HoRo08}, since the pair considered there is $(\frac53,10)$ which is not an $\Lt$--admissible  Strichartz pair, since $q<\frac53<2$. In fact, the norm $\| |u|^4(1+\nabla)u\|_{L^{q'}_{[t,\infty)}L^{r'}_x}$ used in \cite{HoRo08} will only allow pairs $(q,r)$ which are not $\Lt$--admissible Strichartz pairs (the pair $(q',r')$  will not  belong to the $S'(\Lt)$ range). Thus, the original argument in \cite[Proposition 2.2]{HoRo08} had an error. The issue is fixed in \cite{HRerrata} showing that  for $\langle \nabla \rangle=( I-\Delta)^{1/2}$ the $\| \langle \nabla \rangle u\|_{\SLt}$ is bounded,  and thus,  $\|u(t)-e^{it\Delta}\phi^+\|_{H^1}\to 0$ as $t\to +\infty$.
\end{rmk}
\begin{prop}{\rm(Long time perturbation).}\label{longperturbation}
For each $A >0$, there exists $\epsilon_0=\epsilon(A)$ and $c=c(A)$ such that the following holds. Let $u \in H^1_x$ for all $t$ and solve
$\ds i \partial _t u+\Delta u+|u|^4u=0.$
Let $v \in H^1_x$ for all $t$ and define
\begin{align*}
\tilde{e}=i \partial v_t +\Delta v+ |v|^4v.
\end{align*}
If   
$
\| v \|_{\SHs} \leq A, \;  \|\tilde{e}\|_{S'(\dHds)} \leq \epsilon_0 \quad and \quad \| e^{i (t-t_0) \Delta} (u(t_0)-v(t_0)) \|_{\SHs} \leq \epsilon_0,
$
then
$\| u \|_{\SHs}< \infty.$
\end{prop}

\begin{proof}
Define $w=u-v$, then $w$ solves
\begin{align}
\label{equation-LTB}
i w_t +\Delta w + F(v,w)-\tilde{e}=0,
\end{align}
where
$F(v,w)= |w+v|^4(w+v)- |v|^4v $. Since $\| v\|_{\SHs} \leq A$, take a partition of $[t_0, \infty)$ with $N$ subintervals $I_j=[t_j,t_{j+1}]$ that satisfy $\| v\|_{S(\dHs;I_j)} \leq \delta$ for a $\delta$ to be chosen later. Writing the integral equation for \eqref{equation-LTB} in the interval $I_j$, we obtain
\begin{align}
\label{DuhamelsInterval}
w(t)= e^{i (t-t_j) \Delta} w(t_j)+i \int^t_{t_j} e^{i(t-\tau) \Delta} W(\tau) d\tau,
\end{align}
where $W=F(v,w)-\tilde{e}$.

By applying Kato's Strichartz estimate \eqref{Kato-Strichartz} on $I_j$, we obtain
\begin{align*}
\| w \|_{S(\dHs;I_j)} \leq \|  e^{i (t-t_j) \Delta} w(t_j) \|_{S(\dHs;I_j)} +c_1 \| W \|_{S'(\dHds;I_j)},
\end{align*}
where $c_1$ is the constant in  \eqref{Kato-Strichartz}
and 
\begin{align*}
\| W \|_{S'(\dHds;I_j)} & \leq \| F(v,w) \|_{S'(\dHds;I_j)}+\|\tilde{e} \|_{S'(\dHds;I_j)}\\
& \leq \| F(v,w) \|_{L^{\frac{12}{5}}_{I_j} L^{\frac{6}{5}}_x}+\|\tilde{e}\|_{S'(\dHds;I_j)},
\end{align*}
here, the pair $(\frac{12}{5},\frac{6}{5})$ is the conjugate to $(\frac{12}{7},6)$ which is $\dHds$- admissible.
Using H\"older's inequality and a simple fact that $(a+b)^4\leq c\bigl(a^4+b^4\bigr)$, we get
\begin{align*}
\| F&(v,w) \|_{L^{\frac{12}{5}}_{I_j} L^{\frac{6}{5}}_x}
\lesssim \bigl\| |(w+v)-v|(|w+v|^4+|v|^4)\bigr \|_{L^{\frac{12}{5}}_{I_j} L^{\frac{6}{5}}_x}\\
&\lesssim \bigl\| w\bigr\|_{L^{12}_t L^{6}_x}\bigl(\bigl\|w\bigr\|^4_{L^{12}_{I_j} L^{6}_x}+\bigl\|v\bigr\|^4_{L^{12}_{I_j} L^{6}_x}\bigr)
\lesssim \bigl\| w\bigr\|_{\SHs}\bigl(\bigl\|w\bigr\|^4_{S(\dHs;{I_j})}+\bigl\|v\bigr\|^4_{S(\dHs;{I_j})}\bigr).
\end{align*}

Choosing 
$ \delta < \min\bigl\{ 1, \frac{1}{
{4 c_1}} \bigr\}$ and $\|  e^{i (t-t_j) \Delta} w(t_j) \|_{S(\dHs;I_j)}+ c_1 \epsilon_0  \leq \min \bigl\{1, \frac{1}{2\sqrt[4]{4 c_1}}\bigr\}
$,  it follows that
$\| w \|_{S(\dHs;I_j)}  \leq 2 \|  e^{i (t-t_j) \Delta} w(t_j) \|_{S(\dHs;I_j)}+ 2 c_1 \epsilon_0.
$

\noindent Taking  $t=t_{j+1}$, applying $e^{i (t-t_{j+1}) \Delta } $ to both sides of \eqref{DuhamelsInterval} and repeating the Kato estimates, we obtain 
\begin{align}
\| e^{i (t-t_{j+1}) \Delta }w(t_{j+1}) \|_{S( \dHs)} 
\nonumber & \leq 2 \|  e^{i (t-t_j) \Delta} w(t_j) \|_{S(\dHs;I_j)}+ 2 c_1 \epsilon_0.
\end{align}
Iterating this process until $j=0$, we obtain
\begin{align*}
\| e^{i (t-t_{j+1}) \Delta }w(t_{j+1}) \|_{S( \dHs)} & \leq 2^j \| e^{i (t-t_{0}) \Delta }w(t_{0}) \|_{S( \dHs)} +(2^j-1) 2 c_1\epsilon_0 \leq 2^{j+2} c_1 \epsilon_0.
 \end{align*}
 These estimates hold for all intervals $I_j$ for $0 \leq j \leq N-1$, 
 then
$ 2^{N+2} c_1 \epsilon_0 \leq \min\Big\{1, \frac{1}{2\sqrt[4]{4 c_1}}\Big\},
$
which determines how small $\epsilon_0$ has to be taken in terms of $N$ (as well as, in terms of $A$). 
\end{proof} 
{\rmk \label{remark 28}A direct application of \cite[Proposition 2.3]{HoRo08} again is not possible, we would need to estimate $\|v\|^4_{\Lt_tL^\infty_x}$, which is not an $\Lt$-admissible norm in two dimensions. Therefore, we must use a pair $(q,r)$ with $r<+\infty$, which is possible, since  it is not necessary to use a symmetric Strichartz norm $L^q_tL^r_x$ ($q\neq r$) as it was done in \cite[Proposition 2.3]{HoRo08}.}

\subsection{Properties of the Ground State\label{sec:prop GS}}  Pohozhaev identities imply:
\begin{align}
	\| Q \|_{L^6}^6=3\| Q \|_{\Lt}^2
	\label{eq:p2}
\end{align}
(multiply \eqref{ground} by $x\cdot\nabla Q$  and integrate over $x$) and 
\begin{align}
		\| Q \|_{L^6}^6=\| Q \|_{\Lt}^2+ \| \nabla Q \|_{\Lt}^2.
		\label{eq:p1}
\end{align}
(multiply (\ref{ground}) by $Q$ and integrate over $x$).
Substituting  \eqref{eq:p1} and \eqref{eq:p2} into  invariant quantities, we get
\begin{align}
	\| Q \|_{\Lt}\| \nabla Q \|_{\Lt}=\sqrt{2}\| Q \|_{\Lt}^2
	\quad \mbox{and} \quad 
	M[Q]E[Q]=\frac12\| Q \|_{\Lt}^4. \label{eq:p4}
\end{align}
The Gagliardo-Nirenberg estimate and the sharp constant $C_{GN}=\frac{3}{4\|Q\|^4_{\Lt}}$ 
\begin{align}
	\|u\|^6_{L^6}\leq C_{GN}\|u\|_{\Lt}^2\|\nabla u\|^4_{\Lt},
	\label{Gagliardo-Nirenberg}
\end{align}
where $C_{GN}$ is obtained from equality in \eqref{Gagliardo-Nirenberg} with $u$ replaced by $Q,$ see  \cite{We82}.

 \subsection{Properties of the Momentum}\label{momentum}
 
Let $u$ be a solution of \eqref{quintic} with $P[u]\neq0$. Take $\xi_0 \in \Rn$ (chosen later) and let $u_G$ be the Galilean transformation as in \eqref{ugalilean}.  
Noting that $\|\nabla w\|^2_{\Lt}=|\xi_0|^2 M[u]+2\xi_0\cdot P[u]+\|\nabla u\|^2_{\Lt},$ and  that  $M[w]=M[u]$, $ E[w]=\dfrac12|\xi_0|^2M[u]+\xi_0 \cdot P[u]+E[u],$ we minimize the above 
  expressions  to obtain the minimum at $\xi_0=-\frac{P[u]}{M[u]}$, and hence, $P[w]=\xi_0M[u]+P[u]=0$. We also have
\begin{align*}\ds
		E[w]=E[u]-\dfrac{P^2[u]}{2M[u]}
		 \quad\quad \mbox{and}\quad\quad		 
		 \|\nabla w\|^2_{\Lt}=\|\nabla u\|^2_{\Lt}-\dfrac{P^2[u]}{M[u]}.
\end{align*}
Thus, $\ds \ME[w]=\ME[u]-2\Pu^2[u]<1$ and $ \|\nabla w\|^2_{\Lt}\|\ w\|^2_{\Lt}=\|\nabla u\|^2_{\Lt}\|u\|^2_{\Lt}-P^2[u]$. 
Therefore, if $P[w]=0,$  the conditions \eqref{mass-energy},  \eqref{ground-momentum1} and \eqref{ground-momentum2} become 
\begin{align*}
\ds \ME[w]<1,\quad \g_w(0)<1,\quad\text{and}\quad\g_w(0)>1.
\end{align*}
The reduced version of Theorem A is the following

 \begin{thm*}
 Let $u_0\in H^{1}(\Rn)$ and $u(t)$ be the corresponding solution to \eqref{quintic} in $H^1(\Rn)$ with maximal time interval of existence $(-T_*,T^*)$. 
Assume $P[u]=0$ and $\ME[u]<1.$ 
\item[I.] If  
$\g_u(0)<1,$
then
\begin{enumerate}
\item[(a)] $\g_u(t)<1$ for all $t\in \R$, thus, the solution is global in time   
and 
\item[(b)]  $u$ scatters in $H^1(\Rn)$, this means, there is $\phi_\pm\in H^1(\Rn)$ such that 
\begin{align*}
\lim_{t\to \pm\infty}\|u(t)-e^{it\Delta}\phi_\pm\|_{H^1(\Rn)}=0.
\end{align*}
\end{enumerate}
\item[II.] If  $\g_u(0)>1$,  then  $\g_u(t)>1$ for all $t\in(-T_*,T^*)$ and if 
\begin{enumerate}
\item[(a)]   $u_0$ is radial or  $u_0$ is of finite variance, i.e., $|x|u_0 \in \Lt(\Rn)$, then the solution blows up in finite time in both time directions.
\item[(b)]  If $u_0$ non-radial and of infinite variance, then in the positive time direction either the solution blows up in finite time (i.e., $T^* < +\infty $) or there exists a sequence of times $t_n\to +\infty$ such that $\|\nabla u(t_n)\|_{\Lt(\Rn)}\to \infty$. Similar statement holds for $t<0$.
\end{enumerate}
\end{thm*}
\noindent In the rest of the paper we shall assume that $P[u]=0$ and prove Theorem A*.

Observe that bounding the energy $E[u]$ above by the kinetic energy term, we obtain the upper bound in \eqref{thrME}; using the definition of energy and the sharp Gagliardo-Nirenberg inequality \eqref{Gagliardo-Nirenberg} to bound the potential energy term, we obtain a bound from below in \eqref{thrME}, combining, we have
\begin{equation}
2\g^2_u(t)-\g^4_u(t)\leq\ME[u]
\leq 2\g^2_u(t).\label{thrME}
\end{equation}

We plot $y=\ME[u]
$ vs. $\g_u^2(t)$ using the restriction \eqref{thrME} in Figure 1. This plot contains the scenarios for global behavior of solutions  given by  Theorem A*. 
\begin{figure}[htbp]
\begin{flushleft}
\epsfxsize=13 cm \epsfysize=8 cm   
\epsffile{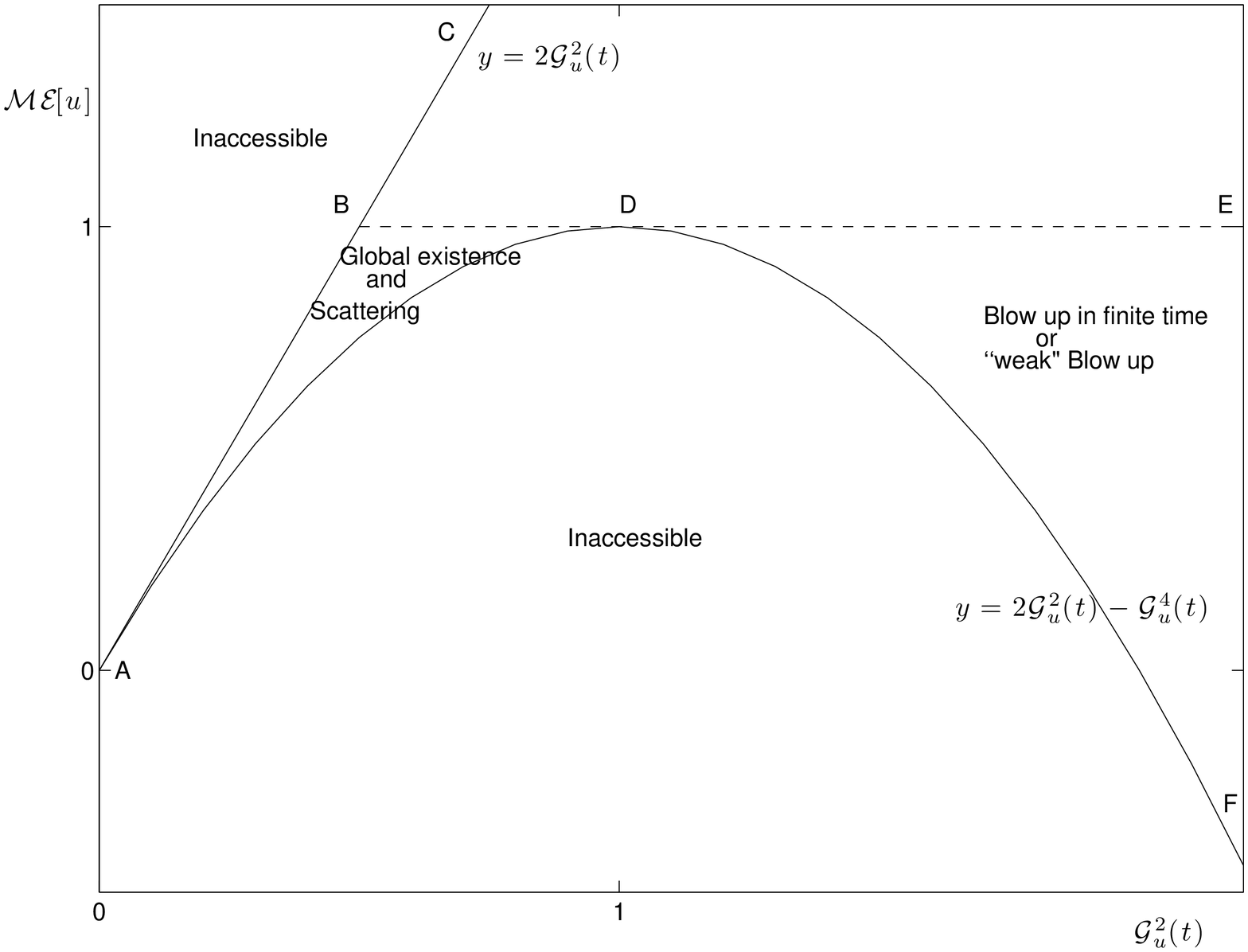}
\end{flushleft}
\caption{Plot of $\ME[u]$
against $\g^2_u(t)$.  The region above the line ABC and below the curve ADF are forbidden regions by \eqref{thrME}. Global existence of solutions and scattering holds in the region ABD, which corresponds to Theorem A* part I. The region EDF explains Theorem A* part II (a) finite time blow up, and the ``weak" blow up fromTheorem A* part II  (b). The characterization of solution on the line BDE and above is an open question. }
 \label{fig1}
\end{figure}

\section{Global versus Blow up Dichotomy}

In this section we discuss the sharp threshold for the  global existence and the finite time blow up of solutions for the $\NLS_5^{+}(\R^2)$. Theorem 2.1 and Corollary 2.5 of Holmer-Roudenko \cite{HoRo07} proved the general case  for the mass-supercritical and energy-subcritical NLS equations with $H^1$ initial data, thus, establishing Theorem A* I(a) and II(a) for finite variance data.  Thus, we only discuss the case of radial initial data in part II(a). First we recall

  \begin{lemma} [Gagliardo-Nirenberg estimate for radial functions \cite{OgTs91}]  \label{radial GN}
Let $u \in H^1(\Rn)$ be  radially symmetric. Then for any $R>0$, $\;u$ satisfies 
\begin{equation}\label{radGN}
	\|u(x)\|^{6}_{L^{6}(R<|x|)} \leq \frac{c}{R^{2}} \| u \|^{4}_{\Lt(R<|x|)}\|\nabla u \|^{2}_{\Lt(R<|x|)},
\end{equation}
and  $c$ is an absolute constant.
\end{lemma}

\begin{proof}[Proof of Theorem A part II(\MakeLowercase{a}) (for radial functions.) ]  
Recall that the variance is given by
 $V(t)=\int |x|^2|u(x,t)|^2dx.$
The standard argument for finite variance data is to examine the derivative of $V$ and show that  $$\partial^2_t V(t)= 32E[u_0]-8\|\nabla u(t)\|_{\Lt}^2<0,$$ which by convexity implies the finite time existence of solutions. To obtain a wider range of blow up solutions, there are more delicate arguments (see \cite{Lush95}, \cite{HoPlRo09}).

Here, for infinite variance radial data, the argument of localized variance is used following  Ogawa-Tsutsumi technique in \cite{OgTs91}.

Let   $\chi \in C^{\infty}(\Rn)$ be radial,
\begin{equation*}
	\chi(r) =\left\{
			\begin{array}{ccc}
			r^2 &  & 0\leq r\leq 1 \\
			\text{~smooth~~} &  & 1<r< 4 \\
			c &  & 4\leq r
			\end{array}\right.
\end{equation*}
such that $\partial^2_r\chi(r)\leq 2$ for all $r\geq0$. Now, for $m>0$ large, let $\chi_m(r)=m^2\chi \bigg(\dfrac rm\bigg)$.  Define the localized variance 
$V(t)=\int \chi(x)|u(x,t)|^2dx$ and consider 
\begin{align}
\partial^2_t V(t)&=4\int \chi'' |\nabla u|^2-\int \triangle^2\chi|u|^2-\frac43\int\triangle\chi|u|^{p+1}\label{eq:varder}.
\end{align}

For $r\leq m$ it follows that $\triangle\chi_m(r)=4$ and $\triangle^2 \chi_m (r)=0$. Each of the three terms in the inequality \eqref{eq:varder}  are bounded as follows:
\begin{eqnarray*}
	4\int \chi_m''|\nabla u|^2&\leq&8 \int_{\Rn} |\nabla u|^2,\\
	-\int \triangle^2\chi_m|u|^2&\leq&\dfrac{c_1}{m^2}\int _{m\leq|x|\leq2m}|u|^2
	\leq\dfrac{c_1}{m^2}\int _{m\leq|x|}|u|^2,\\
	-\int\triangle\chi_m|u|^{p+1}&\leq&-4\int_{\Rn}|u|^{p+1}+c_2\int_{m\leq|x|}|u|^{p+1}.
\end{eqnarray*}
Thus, rewriting \eqref{eq:varder}, we obtain
\begin{align}\notag
\partial^2_t V(t) \leq&
32  E[u]- 8\|\nabla u\|^2_{\Lt} + \frac{c_1}{m^2} \|u\|^2_{\Lt}
+c_3\|u\|^{6}_{L^{6}(|x|\geq m)}\\
 \leq&
32  E[u]- 8\|\nabla u\|^2_{\Lt} + \frac{c_1}{m^2} \|u\|^2_{\Lt}
+\frac{c_4}{m^{2}}\|u\|^{4}_{\Lt}\|\nabla u\|^{2}_{L^{2}}, \label{variance p=5}
\end{align} 
where $\|u\|_{L^{6}(|x|\geq m)}$ was estimated using \eqref{radGN}. 

Let $\epsilon>0$, to be chosen later, pick 
$m_1>\left(\frac{c_1}{\epsilon E[\uQ]}\right)^{\frac12}\|u\|_{\Lt}$,  
$\; m_2>\left(\frac{c_4}{\epsilon}\right)^{\frac12}\|u\|^{2}_{\Lt}\;$ and $m=\max\{m_1,m_2\}$ to get
\begin{align}\label{variance-blow}
\partial^2_t V(t)
<32E[u]
-(8-\epsilon)\|\nabla u\|^2_{\Lt}
+\epsilon E[\uQ].
\end{align}

Furthermore, the assumptions $\ME[u]<1$ and  $\g_{u}(0)>1$ imply that there exists $\delta_1>0$ such that 
$
\ME[u]<1-\delta_1
$
and there  exists  $\delta_2=\delta_2(\delta_1)$ such that  
$
\g_{u}(t)>(1+\delta_2)$ for all $t\in I$. 
Multiplying  both sides of \eqref{variance-blow} by $M[u_0]$, leads to
\begin{align*}
M[u_0]\partial^2_t V(t)
<&32(1-\delta_1)M[\uQ]E[\uQ]-(8-\epsilon)(1+\delta_2)\|\uQ\|_{\Lt}^2\|\nabla \uQ\|^2_{\Lt}+\epsilon M[\uQ] E[\uQ]\\
<&[32(1-\delta_1)-4(8-\epsilon)(1+\delta_2)+\epsilon] M[\uQ] E[\uQ],
\end{align*}
the last inequality follows since $4E[\uQ]=\|\nabla \uQ\|^2_{\Lt}$. Choosing $\epsilon<\frac{32(\delta_1+\delta_2)}{5+4\delta_2}$ implies that
 the second derivative of the  variance is bounded by a negative constant for all $t\in \R$, i.e.,
 $\partial^2_t V(t)<-A,$ and integrating twice over t, we have that 
$V(t)<-At^2+Bt+C.$
Thus, there exists $T$ such that $V(T)<0$ which is a contradiction. Therefore, radially symmetric  solutions of the type described in Theorem A* part II (a) must blow up in finite time.
\end{proof}

In the rest of this section we establish useful estimates on solutions with initial gradient $\g_u(0)<1$.
\begin{lemma}{\rm(Lower bound on the convexity of the variance).}
\label{Lower-bound-convexity}
Let $u_0 \in H^1$ satisfy \eqref{mass-energy} and $\g_u(0)<1$. 
Assume $\delta>0$ such that
$
\ME[u_0]< (1-\delta).  
$

If $u$ is a solution to \eqref{quintic} with initial data $u_0$, then there exists $c_{\delta} >0$ such that for all $t \in \R,$
\begin{align*}
32 E[u]-8 \| \nabla u(t) \|_{\Lt}^2 =8 \| \nabla u(t) \|_{\Lt}^2-\frac{16}{3}  \|u\|_{L^6}^6 \geq c_{\delta} \| \nabla u(t) \|_{\Lt}^2,
\end{align*}
in other words, for finite variance solutions, $\partial_t v(t)\geq c \delta_2\| \nabla u(t) \|_{\Lt}^2 $.
\end{lemma}
\begin{proof}
From  the proof of Theorem A* part I,  for  $\delta>0$, there exists a $\delta_1(\delta)>0$ such that $
\g_u^2(t) \leq (1-\delta_1)^2 
$ for all $t \in \R$. 
Let
\begin{equation} 
h(t)= \big(8 \| \nabla u \|_{\Lt}^2-  \frac{16}3 \|u\|_{L^6}^6 \big)\frac{\|u\|^2_{\Lt}}{\|Q \|^2 _{\Lt}\| \nabla Q \|^2_{\Lt}}\label{hyp}.
\end{equation}
By the Gagliardo-Nirenberg inequality \eqref{Gagliardo-Nirenberg} and the exact value of $C_{GN}$,  we get
$\ds h(t) \geq  8 \g_u^2(t) \big( 1-\g_u^2(t) \big).$
Setting  $g(y)=y^2(1-y^2)$, it follows
$h(t) \geq 8 g\big(\g_u(t)\big).$
We only consider $g(y)$ in the range $[0 , 1-\delta_1]$. Thus, $g(y) \geq c_{\delta} y^2$, obtaining the result.
\end{proof}

\begin{lemma}{\label{equivalence grad and energy}}{\rm(Equivalence of energy with the gradient).}
Let $u_0 \in H^1$ satisfy $\g_u(0)<1$ and $\ME[u_0]<1$. Then
\begin{align}\label{equiv-energy}
\frac{1}{4}\| \nabla u \|^2_{\Lt} \leq E[u] \leq \frac1 2 \| \nabla u \|^2_{\Lt}.
\end{align}
\begin{proof}
The first inequality is obtained by observing that the Gagliardo- Nirenberg inequality and the value of $C_{GN}$ \eqref{Gagliardo-Nirenberg}, the  Pohozhaev identity \eqref{eq:p4} and the hypothesis $\g_u(0)<1$ yield
\begin{align*}
E[u]&=\frac1 2 \| \nabla u \|^2_{\Lt} - \frac{1}{6} \| u \|_{L^6}^6
\geq\frac12\| \nabla u \|^2_{\Lt}\left(1-\frac{C_{GN}}6\|  u \|^2_{\Lt}\| \nabla u \|^2_{\Lt}\right)\\
  &\geq\frac12\| \nabla u \|^2_{\Lt}\left(1-\frac{\|  Q \|^2_{\Lt}\| \nabla Q \|^2_{\Lt}}{8\|  Q \|^4_{\Lt}}\right)= \frac{1}{4}  \| \nabla u \|^2_{\Lt}
,\end{align*}
and the second inequality trivially follows from the definition of energy.
\end{proof}
\end{lemma}
\begin{cor}{\label{cor grad and energy}}
Let $u_0 \in H^1$ satisfy $\g_u(0)<1$ and  $\ME[u_0]<1$,  
then for all $t,$ $ \omega=\sqrt{\ME[u]}$, $\g_u(t)\leq \omega$, and 
$$16(1-\omega^2) E[u] \leq 8(1-\omega^2)\| \nabla u \|^2_{\Lt}\leq 8 \| \nabla u \|_{\Lt}^2- \frac{16}3 \|u\|_{L^6}^6. 
$$
\end{cor}
\begin{proof}
By the left inequality of \eqref{equiv-energy}, $\| \nabla u \|^2_{\Lt} \leq 4 E[u].$ Multiplying by mass of $u_0$ normalized by $\| Q \|^2_{\Lt} \| \nabla Q \|^2_{\Lt}$ and using that $\| \nabla Q \|^2_{\Lt}=4 E[Q]$, we obtain
$\g_u(t)\leq\omega.
$ 
Thus, 
\begin{align*}
V(t)&=8\| \nabla u \|^2_{\Lt}-\frac{16}3\|  u \|^6_{L^6}\leq 8\| \nabla u \|^2_{\Lt}\left(1-\frac{2C_{GN}}3\|  u \|^2_{\Lt}\| \nabla u \|^2_{\Lt}\right)\\
&\leq 8\| \nabla u \|^2_{\Lt}\left(1-\g_u(t)
\right)=8(1-\omega^2)\| \nabla u \|^2_{\Lt}.
\end{align*}
The above estimate is obtained by combining  the variance,  Gagliardo-Nirenberg inequality \eqref{Gagliardo-Nirenberg},  the exact value of $C_{GN}$, the  Pohozhaev identity 
and the estimate $\g_u(0)<w$, and applying  Lemma \ref{equivalence grad and energy}, we obtain the left inequality, which completes the proof.
\end{proof}


\begin{prop}{\rm(Existence of wave operator).}\label{Existence of wave operator}
Let $\psi \in H^1(\Rn).$

\noindent I. Then there exists $v_{+}\in H^1$ such that for some $-\infty<T^*<+\infty$ it produces a solution $v(t)$ to $\NLSf$ on time interval $[T^*,\infty)$ such that 
\begin{equation}\label{wave1}
\|v(t)-e^{it\Delta}\psi\|_{H^1}\to0 \quad\quad\text{~~as~~}\quad\quad t\to+\infty.
\end{equation}
 Similarly, there exists $v_{-}\in H^1$ such that for some $-\infty<T_*<+\infty$ it produces a solution $v(t)$ to $\NLSf$ on time interval $\;(-\infty,T_*]$ such that 
\begin{equation}\label{wave2}
\|v(-t)-e^{-it\Delta}\psi\|_{H^1}\to0 \quad\quad\text{~~as~~}\quad\quad t\to+\infty.
\end{equation}

\noindent II. Suppose that for some $0<\sigma<1$
\begin{equation}
\label{psi-identity}
\frac12\| \psi \|^{2}_{\Lt} \| \nabla \psi \|^{2}_{\Lt}  < \sigma^2 M[\uQ] E[\uQ]\;.
\end{equation}
Then there exists $v_0 \in H^1$ such that $v(t)$ solving $\NLSf$ with initial data $v_0$ is global in $H^1$ with 
\begin{eqnarray}
 M[v]= \| \psi \|_{\Lt}^2,\quad\quad
 E[v]= \frac{1}{2} \| \nabla \psi \|_{\Lt}^2, \quad\quad 
 \g_v(t)\leq\sigma< 1\label{wave meg}\quad \text{~~and~~}\\
 \quad\quad\quad\quad
\| v(t) - e^{i t \triangle} \psi \|_{H^1}\to0 \quad\quad\text{~~as~~}\quad \quad t\to\infty.\label{wave scat}
\end{eqnarray}
Moreover, if $\| e^{i t \triangle} \psi \|_{S(\dHs)} \leq \delta_{sd},$ then
$\| v_0 \| _{\dHs} \leq 2 \| \psi \|_{\dHs}$   and  $
\| v \|_{\dH^{s}}  \leq 2 \| e^{i t \triangle} \psi \|_{S(\dHs)}.
$
\end{prop}

\begin{proof}
\noindent I.  This is essentially Theorem 2 part (a) of Strauss \cite{St81} adapted to the case $s=\frac12$ ($d=2$ and $p=5$) (see also Remark (36) and \cite[Theorem 17]{St81s}).

\noindent II. For this part, 
we want to find a solution to the integral equation
\begin{equation}
v(t)=e^{it\Delta}\psi_+-i\int_t^{+\infty}e^{i(t-\tau)\Delta}\big(|v|^4v\big)(\tau)d\tau.\label{int-wave}
\end{equation}

Note that for $T>0$ from the small data theory (Proposition \ref{small data}) there exists $\delta_{sd}>0$ such that $\|e^{it\Delta}\psi_+\|_{\SHSTx}\leq\delta_{sd}$.  
Thus, repeating the argument of Proposition \ref{small data} we first show that we can solve the equation \eqref{int-wave} in $\dH^{1/2}$ for $t\geq T$ with T large. So this solution $v(t)$ is in $\dH^{1/2}$, and hence, we have that $\|v\|_{S(\dH^{1/2};[T,+\infty))}$ is small for large $T>>0$. Now we will estimate $\|\nabla v\|_{\SLTx}$, which will also show that $v$ is in $H^1$.
Observe that for any $u \in H^1$  
\begin{align*}
\| \nabla(|u|^4u)\|_{S'(\Lt)}&\leq \| \nabla(|u|^4u)\|_{L^{\frac{3}{2}}_tL^{\frac{6}{5}}_x}\lesssim \|u\|_{L^8_tL^8_x}^4 \|\nabla u\|_{L^6_tL^3_x} \lesssim \|u\|_{\SHs}^4 \|\nabla u\|_{\SLt}.
\end{align*}
Applying the Strichartz estimates (\ref{stri}) and Kato-Strichartz estimate (\ref{Kato-Strichartz}) yields
\begin{align*}
\|\nabla v\|_{\SLTx}&\leq c_1\|\nabla\psi_+\|_{\Lt}+c_2\big\|\nabla\big(|v|^4v\big)\big\|_{\dSLTx}\\
&\leq  c_1\|\psi_+\|_{\dH^1}+c_3 \|\nabla v\|_{\SLTx}\| v\|^4_{\SHSTx}.
\end{align*}
Since $T$ can be chosen large, so that $c_3 \| v\|^4_{\SHSTx}\leq\frac12,$ we get 
$$\|\nabla v\|_{\SLTx}\leq2 c_1\| \psi_+ \|_{\dH^1}.$$ 
Using this fact, we also get $\|\nabla (v(t)-e^{it\Delta}\psi_+)\|_{\SLTx}
\leq c\|\psi_+\|_{\dH^1},
$ since
\begin{align*}
\|\nabla (v(t)-e^{it\Delta}\psi_+)\|_{\SLTx}&\leq c \|\nabla v\|_{\SLTx}\| v\|^4_{\SHSTx}+c\|\psi_+\|_{\dH^1}.
\end{align*}
Thus,
$$
\lim_{T\to+\infty}\|\nabla (v(t)-e^{it\Delta}\psi_+)\|_{\SLTx}=0.
$$  
So we showed that as $t\to+\infty$,  $v(t)\to e^{it\Delta}\psi_+$ in $H^1$. In particular, this means that $v(t)\to e^{it\Delta}\psi_+$ in $\Lt$, hence $$M[v]\equiv\|v(t)\|^2_{\Lt}=\|e^{it\Delta}\psi_+\|^2_{\Lt}=\|\psi_+\|^2_{\Lt}.$$

Moreover,   Sobolev embedding implies $e^{it\Delta}\psi_+\to0$ in $L^6$. Thus, $\|\nabla e^{it\Delta}\psi_+\|_{\Lt}$ is bounded, and 
\begin{align*}
E[v]=\lim_{t\to+\infty}\Big(\dfrac12 \|\nabla e^{it\Delta}\psi_+\|^2_{\Lt}-\dfrac16\|e^{it\Delta}\psi_+\|^6 _{L^6}\Big)=\dfrac12 \|\nabla\psi_+\|^2_{\Lt}.
\end{align*}

From the hypothesis \eqref{psi-identity}, we obtain
$$M[u]E[u]=\frac12\| \psi_+ \|^{2}_{\Lt} \| \nabla \psi_+ \|^{2}_{\Lt}  < \sigma^2 M[\uQ] E[\uQ]\;$$
 and so $\ME[u]<1$. Furthermore, $\|\nabla v(t)\|^2_{\Lt}=\|\nabla e^{it\Delta}\psi_+\|^2_{\Lt}=\|\nabla \psi_+\|^2_{\Lt}$, and so,

\begin{align*}
\lim_{t\to+\infty}
\|\nabla v(t)\|^2_{\Lt}\| v\|^2_{\Lt}&=\lim_{t\to+\infty}
\|\nabla e^{it\Delta}\psi_+\|^2_{\Lt}\| e^{it\Delta}\psi_+\|^2_{\Lt}=
\|\nabla \psi_+\|^2_{\Lt}\| \psi_+\|^2_{\Lt}\\
&<2\sigma^2{M[Q]E[Q]}=
\sigma^2\|\nabla Q\|^2_{\Lt}\| Q\|^2_{\Lt}.
\end{align*}
Thus, $$\lim_{t\to\infty}\g_v(t)\leq \sigma<1.$$

For sufficiently large  $T>0$, we can get that $\g_v(T)<1.$ Now we are in the assumption of  Theorem A* part I (a),  which shows that $v(t)$ exists globally  and  evolving  it 
from $T$ back to 0, we will obtain the data $v_0\in H^1$ as desired.
\end{proof}

\section{Outline of Scattering via Concentration Compactness}\label{scattering}
The goal of this section is to outline the proof of scattering in $H^1$ for the global solution of  \eqref{quintic}, i.e., Theorem A (I. part b). The proof of the main steps will be given in Sections  \ref{profile section} and \ref{step 1and 2}.

\begin{definition}
Suppose $u_0\in H^1$ and  let $u$ be the corresponding $H^1$ solution to \eqref{quintic} and $[0,T^*)$ be the maximal (forward in time) interval of existence. We say that  $SC(u_0)$ holds if $T^*=+\infty$ and $\|u\|_{\SHs}<\infty$. Note that if $SC(u_0)$ holds, then together with  Proposition \ref{H^1 Scattering} we obtain $H^1$ scattering of $u(t)=\NLS(t)u_0$. 
\end{definition}
Our goal is  to prove  the following: if $\g_u(0)<1$ and $\ME[u]<1$, then $SC(u_0)$ holds.

The hypotheses give an {\it a priori} bound  for $\|\nabla u(t)\|_{\Lt}$ (by Theorem A part I), thus, the maximal forward time of existence is $T=+\infty$. Therefore, it remains to show that the global-in-time $\dHs$ Strichartz norm is finite, i.e., $\|u\|_{\SHs}<\infty$. We prove this using the induction argument on the mass-energy threshold as in \cite{KeMe06}, \cite{HoRo08}.
\medskip

\noindent{\bf Step 0: Small Data.} The equivalence of energy with the gradient from Lemma \ref{equivalence grad and energy} yields 
$$ 
\|u_0\|_{\dHs}^6\leq(\|u_0\|_{\Lt}\| \nabla u_0\|_{\Lt})^3 \leq (4M[u]E[u])^{3/2}.
$$
If $\g_u(0)<1$ and $M[u]E[u]<\frac14\delta_{sd}^4$, then 
$\|u_0\|_{\dHs}\leq\delta_{sd}$
and  
$\|e^{it\Delta}u_0\|_{\SHs}\leq c \delta_{sd}$ by Strichartz estimates. Thus,  the small data (Proposition \ref{small data}) yields  $SC(u_0)$ condition.

This observation gives the basis for induction: we assume $\g_u(0)<1.$ Then for small $\delta>0$ such that $M[u_0]E[u_0]<\delta$, $SC({u_0})$ holds. 

Define the supremum of all such $\delta$ for which $SC(u_0)$ holds, namely,
\begin{align*}
(ME)_c= \sup \big\{\delta \;|\; u_0\in H^1& \mbox{ with the property: } \\
&\g_u(0)<1 \mbox{ and }M[u]E[u] < \delta \Rightarrow SC(u_0) \mbox{ holds}\big\}.
\end{align*}

We want to show that $(ME)_c =M[Q]E[Q]$. Observe that  $u_0(x)=Q(x)$ does not scatter, and this is the solution such that $\g_Q(0)=1$ and $M[u]E[u]=M[Q]E[Q]$. To be precise, one should consider $\g_u(0)\leq1$ in the definition of $(ME)_c$,  instead of the strict inequality $\g_u(0)<1.$ However, $\g_u(0)$=1 only when $\ME[u]=1$ 
 (see Figure \ref{fig1} point D), thus, it suffices to consider the strict inequality $\g_u(0)<1.$

Assume that $(ME)_c<M[Q]E[Q]$.\medskip

\noindent{\bf Step 1: Induction on the scattering threshold and construction of the ``critical'' solution.} 
Since $(ME)_c< M[Q]E[Q]$, we can find a sequence of initial data $\{u_{n,0}\}$ in $H^1$ which will approach the threshold $(ME)_c$ from above and produce solutions which do not scatter, namely, there exists a sequence $\{u_{n,0}\}\in H^1$ producing the NLS solution $u_n(t)=\NLS(t)u_{n,0}$ with 
 \begin{align}\label{sequn}
 \g_{u_n}(0)<\sigma\; \mbox{and}\; M[u_{n,0}]E[u_{n,0}]\searrow(ME)_c\;\mbox{as}\; n\to\infty
 \end{align}
  and $\|u_n\|_{\SHs}=+\infty$ (this is possible by definition of supremum of $(ME)_c$), i.e., $SC(u_{n,0})$ does not hold.

This sequence will allow us to construct (via profile decompositions) a ``critical" solution of $\NLSfp$, denoted by $u_c(t),$ that will lie exactly at the threshold $(ME)_c$ and will not scatter, see Proposition \ref{existence of u_c}.
\medskip

\noindent{\bf Step 2:}  {\bf Localization properties of the critical solution.} The critical solution $u_c(t)$ will have the property that it is precompact in $H^1,$ namely,  $K=\{u_c(t)|t\in[0,+\infty)\}$ is precompact in $H^1$ (Lemma \ref{precompact}), and its localization  implies that 
for given $\epsilon >0$, there exists an $R>0$ and some path $x(t)$ such that
$
\|\nabla u(x,t)\|_{\Lt(|x+x(t)|>R)}^2\leq\epsilon
$   
uniformly in $t.$ This combined with the zero momentum will give control on the growth of $x(t)$ (Lemma \ref{spacial translation}). Note that in the radial case $x(t)\equiv0.$ On the other hand, such compact in $H^1$ solutions  with the control on $x(t)$, can only be zero solutions, by  
the rigidity theorem (Theorem \ref{rigidity}), which contradicts the fact that $u_c$ does not scatter. Therefore, such $u_c$ does not exist and the assumption that $(ME)_c<M[Q]E[Q]$ is not valid. This finishes the proof of scattering in Theorem A*.

 In section \ref{profile section} we proceed with the linear and nonlinear profile decompositions and in section \ref{step 1and 2} we give the proof of claims in Step 1 and Step 2. 

\section{Profile decomposition}\label{profile section}

This subsection contains the profile decomposition for linear and nonlinear flows for $\NLS^+_5(\Rn)$, analogous to the Keraani \cite{Ke01}, and a reordering of the decompositions that will be used in the proof of the ``weak" blow up. 


\begin{prop}{\rm(Linear Profile decomposition.)}\label{nonradialPD}
Let $\phi_n(x)$ be a uniformly bounded  sequence in $H^1$. Then for each $M\in\N$ there exists a subsequence of $\phi_n$ (also denoted $\phi_n$), such that, for each $1 \leq j \leq M$, there exist, fixed in $n$, a  profile $\psi^j$ in $H^1$, a sequence  $t^j_n$ of time shifts, a sequence $x^j_n$ of space shifts 
and a sequence $W^M_n(x)$ of remainders\footnote{Here, $W^M_n(x)$ and  $\wt^M_n(x)$ represent the remainders for the linear and nonlinear profile decompositions, respectively.} 
 in $H^1$,  such that  
$\ds
\phi_n(x)=\sum^{M}_{j=1} e^{-i t_n^j \Delta } \psi^j(x-x^j_n) + {W}^M_n(x)
$
with the properties:
\begin{itemize}
\item Pairwise divergence for the time and space sequences. 
For $1 \leq k \neq j \leq M$, 
\begin{align}
\label{time-space-seq}
\lim_{n \to \infty} |t^j_n-t^k_n | + |x^j_n - x^k_n|=+ \infty.
\end{align} 
\item  Asymptotic smallness for the remainder sequence
\begin{align}
\label{smallnessProp}
\lim_{M \to \infty} \big( \lim_{n \to \infty} \| e^{i t \Delta}W^M_n \|_{\SHs}\big ) = 0.
\end{align}
\item Asymptotic Pythagorean expansion. 
For fixed $M\in\N$ and any $0 \leq s \leq 1$, we have 
\begin{align}
\label{limHs}
\| \phi_n \|^2_{\dot{H}^s} = \sum^M_{j=1} \| \psi^j \|^2_{\dot{H}^s} + \|W^M_n \|^2_{\dot{H}^s} +o_n(1).
\end{align}
\end{itemize}
\end{prop}

\begin{proof} 
Let $\phi_n$ be uniformly bounded in $H^1$, i.e., there exists $0<c_1$ such that   $\|\phi_n\|_{H^1}\leq c_1$.
Let $(q,r)$ be $\dHs$ admissible pair. Interpolation and Strichartz estimates with $\theta=\frac{2}{r-2}$, ($0<\theta<1$), $r_1=2r$, and $q_1=\frac{4r}{r-2}$  yield\footnote{One could choose $r_1=kr$, for $k>1,$ thus, $q_1=\frac{2kr}{kr-2}$ and $\theta=\frac{4(k-1)}{kr-4}$ and $0<\theta<1$, however, the choice of $r_1=2r$ is analogous with \cite{HoRo08}.}
\begin{align}
\| e^{it \Delta} W^M_n \|_{L^q_t L^r_x}
&\leq \|e^{it \triangle} W^M_n \|^{1-\theta}_{L^{q_1}_t L^{r_1}_x} \| e^{it \triangle} W^M_n \|^{\theta}_{L^{\infty}_t L^4_x}. 
\label{eq:rem}
\end{align}
The goal is to decompose a profile $\phi_n$  as $\sum^{M}_{j=1} e^{-i t_n^j \Delta } \psi^j(x-x^j_n) + {W}^M_n(x)$ with  $\|W^M_n(x)\|_{\dHs}\leq c_1$. Since \eqref{eq:rem} holds, it suffices to show 
$$
\lim_{M\to +\infty} \left[ \limsup_{n\to +\infty}\|e^{it\Delta}
W_n^M\|_{L_t^\infty L_x^4}\right] = 0 \, .
$$
 
\noindent{\it Construction of  $\psi^1$ }: 
 Let $A_1 = \limsup_{n\to+\infty} \| e^{i t \Delta } \phi_n \|_{L^{\infty}_t L_x^4}$. 
If $A_1=0$, we are done by taking $\psi^j=0$ for all $j$.  
Suppose that $A_1>0$, and 
$c_1= \limsup_{n\to+\infty}\| \phi_n \|_{H^1}.$ 
Passing to a subsequence $\phi_n$, it is shown\footnote{Since the $\psi^j$ are constructed inductively as in the proof of \cite[Lemma 5.2]{HoRo08} we omit the details.} that there exist sequences $t^1_n$ and $x^1_n$ and a function $\psi^1 \in H^1$, such that $e^{it^1_n \Delta } \phi_n( \cdot + x^1_n)  \rightharpoonup \psi^1$  in $H^1,$ such that
\begin{align}
8 c_1^2 \| \psi^1 \|_{\dHs} \geq A_1^3\label{boundpsi}.
\end{align}
  
  Define $W^1_n (x) = \phi_n(x)-e^{-it^1_n \Delta} \psi^1(x-x^1_n)$. Observe that $e^{it^1_n \Delta } \phi_n( \cdot + x^1_n)  \rightharpoonup \psi^1$ in $H^1$, for any $0\leq s \leq1$ it follows
  $
  \langle  \phi_n,e^{-it^1_n \Delta} \psi^1\rangle_{\dH^s}=\langle e^{it^1_n \Delta}\phi_n, \psi^1 \rangle_{\dH^s}\to\|\psi^1\|_{\dH^s}^2.
  $
  Since $ \|W_n^1 \|^2_{\dH^s}=\langle  \phi_n-e^{-it^1_n \Delta} \psi^1, \phi_n-e^{-it^1_n \Delta} \psi^1 \rangle^2_{\dH^s}$, we have
\begin{align*}
\lim_{n \to \infty} \| W^1_n \|^2_{\dH^s} =
\lim_{n \to \infty} \| e^{i t^1_n \Delta} \phi_n \|^2_{\dH^s} -\| \psi^1 \|^2_{\dH^s}.
\end{align*} 
Thus, taking $s=1$ and $s=0$ yields  $\|W^1_n\|_{H^1}\leq c_1.$
\medskip

\noindent{\it Construction of $\psi^j$ for $j \geq 2$} :
 Inductively  $\psi^j$ are constructed from $W^{j-1}_n$.   
  Let $M \geq 2$.   Suppose that $\psi^j$, $x_n^j$, $t^j_n$ and $W^j_n$ are known for 
$j \in \{ 1 ,\cdots , M-1 \}$. Consider
$\ds
A_M= \limsup_n \| e^{it\Delta}\phi^{M-1}_n \| _{L^{\infty}_t L_x^4}.
$
If $A_M=0,$ then we are done (by taking  $\psi^j=0$ for $j \geq M$). 
Assume $A_M>0$.
 Apply the previous step to $W^{M-1}_n$,  and let 
$c_M= \limsup_n \| W^{M-1}_n \|_{H^1},$ 
thus,
we obtain sequences (or subsequences) $x_n^M, t_n^M$ and a function $\psi^M \in H^1$ such that
\begin{align}
e^{it^M_n \Delta } W^{M-1}_n( \cdot + x^M_n)  \rightharpoonup \psi^M \mbox{  in  }  H^1\quad \mbox{ and }\quad
  8 c_M^2 \| \psi^M \|_{\dot{H}^{\frac{1}{2}}} \geq  A_M^3.
 \label{condpsi} 
  \end{align}

Define 
$
\ds W^M_n(x) = W^{M-1}_n (x) - e^{-i t^M_n \Delta} \psi^M (x-x^M_n).
$
Then (\ref{time-space-seq}) and  (\ref{limHs})  follow from induction, i.e., assume (\ref{limHs}) holds at rank $M-1$. Expanding
$
\|W^M_n\|^2_{\dH^s}= \| e^{i t^M_n \Delta} W_n^{M-1} (\cdot+x^M_n)-\psi^M\|^2_{\dH^s}
$ 
and applying  the weak convergence, yields (\ref{limHs}) at  rank $M.$ 

To show condition (\ref{time-space-seq}), assume the statement is true for $j,k \in \{ 1, \hdots , M-1\}$, that is $|t^j_n-t^k_n | + |x^j_n - x^k_n|\to+ \infty $. Take $k \in \{ 1, \hdots , M-1\}$, we want to show that 
$
|t^M_n-t^k_n | + |x^M_n - x^k_n|\to+ \infty.
$
 Passing to a subsequence, assume $t^M_n-t^k_n\to t^{M_1}$ and $x^M_n-x^k_n\to x^{M_1}$ are finite. Then as $n \to \infty$
 \begin{align*}
e^{i t_n^M \Delta} W^{M-1}_n (x+x_n^M) = e^{i (t^M_n-t^j_n) \Delta}(e^{i t_n^j \Delta} &W^{j-1}_n (x+x_n^j)-\psi^j(x+x_n^j))
\\&
- \sum^{M-1}_{k=j+1} e^{i (t_n^j-t^k_n) \Delta} \psi^{k}(x+x_n^j-x_n^k).
 \end{align*}
The orthogonality condition (\ref{time-space-seq}) implies that  the right hand side goes to $0$ weakly in $H^1$, while the left side converges weakly to $\psi^M$, which is nonzero, contradiction. Then the orthogonality condition (\ref{time-space-seq}) holds for $k=M$. Since (\ref{limHs}) holds for all $M,$ we have 
$
\| \phi_n \|^2_{\dot{H}^s} \geq \sum^M_{j=1} \| \psi^j \|^2_{\dot{H}^s} + \|\,W^M_n \|^2_{\dot{H}^s}. 
$
Thus,  $c_M\leq c_1$. Taking $s=1/2$, and the fact that for all $M,$ $A_M>0,$ yields together with  (\ref{condpsi})
$$\sum _{M\geq1}\Big(\frac{A_M^3}{8c^2_1}\Big)^2\leq \sum _{n\geq1}\|\psi^M\|^2_{\dHs}\leq \limsup_n\|\phi_n\|^2_{\dHs}\leq \infty.$$
 Therefore, $A_M \to 0$ as $M \to \infty,$ which implies (\ref{smallnessProp}).
 \end{proof}

\begin{prop}\label{energy Pythagorean expansion}{\rm (Energy Pythagorean expansion).} Under the hypothesis of Proposition \ref{nonradialPD}, we have 
\begin{align}
\label{pythagorean}
E[\phi_n]=\sum_{j=1}^M [e^{-it^j_n\Delta}\psi^j]+E[W^M_n]+o_n(1).
\end{align}
\end{prop}
\begin{proof}
By definition of $E[u]$ and (\ref{limHs}) with $s=1$, it suffices to prove that  for all $M\leq1$, we have
\begin{align}
\label{norm6seq}
\|\phi_n\|^6_{L^6}=\sum_{j=1}^M\|e^{-it_n^j\Delta}\psi^j\|^6_{L^6}+o_n(1).
\end{align}
{\it Step 1. Pythagorean expansion of  a sum of orthogonal profiles.} Fix $M\geq1$. We want to show that the condition (\ref{time-space-seq}) yields
\begin{align}
\label{normsum}
\bigg\|\sum_{j=1}^M e^{-it^j_n\Delta}\psi^j(\cdot-x^j_n)\bigg\|^6_{L_x^6}=\sum_{j=1}^M\|e^{-it_n^j\Delta}\psi^j\|^6_{L_x^6}+o_n(1).
\end{align}
By rearranging and reindexing, we can find $M_0\leq M$ such that 
\begin{enumerate}
\item[(a)]  $t_n^j$ is bounded in $n$ whenever $1\leq j\leq M_0,$ 
\item[(b)] $|t_n^j|\to \infty$ as $n\to\infty$  if $M_0+1\leq j\leq M.$
\end{enumerate}

For case (a) take a subsequence and assume that  for each $1\leq j \leq M_0$, $t_n^j$ converges (in $n$), then adjust the profiles $\psi^j$'s such that we can take $t^j_n=0$. From (\ref{time-space-seq})  we have
$|x^j_n-x^k_n|\to +\infty$ as $n\to\infty$, which implies
\begin{align}
\label{normsum0}
\bigg\|\sum_{j=1}^{M_0} \psi^j(\cdot-x^j_n)\bigg\|^6_{L_x^6}=\sum_{j=1}^{M_0}\|\psi^j\|^6_{L_x^6}+o_n(1),
\end{align}

For case (b), i.e., for $M_0\leq k \leq M$,  $|t_n^k|\to \infty$ as $n\to\infty,$ take $\tilde{\psi}\in \dH^{5/6}\cap L^{6/5}$, thus, the Sobolev embedding and the $L^p$ space-time decay estimate yield 
\begin{align*}
\|e^{-it^k_n\Delta}\psi^k\|_{L_x^6}\leq c\|\psi^k-\tilde{\psi}\|_{\dH^{5/6}}+\dfrac {c}{|t^k_n|^{2/3}}\|\tilde{\psi}\|_{L^{6/5}},\end{align*}
and approximating $\psi^k$ by $\tilde{\psi} \in C^\infty_c$ in $\dH^{5/6}$,  we have 
\begin{align}
\label{divergent}
\|e^{-it^k_n\Delta}\psi^k\|_{L_x^6}\to 0 \mbox{    as } n\to\infty.
\end{align}
Thus, combining (\ref{normsum0}) and (\ref{divergent}), we obtain (\ref{norm6seq}).

{\it Step  2. Finishing the proof.}  Note that 
\begin{align*}
\|W^{M_1}_n\|_{L^6_x}&\leq\|W^{M_1}_n\|_{L^\infty_tL^6_x}
\leq\|W^{M_1}_n\|^{1/2}_{L^\infty_tL^4_x}\|W^{M_1}_n\|^{1/2}_{L^\infty_tL^{12}_x}\\
&\leq\|W^{M_1}_n\|^{1/2}_{L^\infty_tL^4_x}\|W^{M_1}_n\|^{1/2}_{L^\infty_t\dH^1_x}
\leq\|W^{M_1}_n\|^{1/2}_{L^\infty_tL^4_x}\sup_n\|\phi_n\|^{1/2}_{H^1},
\end{align*}
where in the last line we used the embeddings $\dH^{1}\hookrightarrow \dH^{5/6}\hookrightarrow L^{12}$ on $\Rn$.
Thus, by (\ref{smallnessProp}) it follows that
\begin{align}
\lim_{M_1\to+\infty}\Big(\lim_{n\to+\infty} \|e^{it\Delta}W^{M_1}_n\|_{L^6}\Big)=0.
\label{limM1}
\end{align}
Let $M\geq1$ and $\epsilon>0.$ The sequence of profiles $\{\psi^n\}$ is uniformly bounded in $H^1$ and in $L^6$. Thus, (\ref{limM1}) implies  the sequence of remainders $\{W^M_n\}$ is also uniformly bounded in $L_x^6$. Thus, pick $M_1\geq M$ and $N_1$ such that for $ n\geq N_1$, we have
\begin{align}
\label{est1}
\Big|\|\phi_n&-W^{M_1}_n\|_{L^6_x}^6-\|\phi_n\|_{L^6_x}^6\Big|+\Big|\|W^{M}_n-W^{M_1}_n\|_{L^6_x}^6-\|W^{M}_n\|_{L^6_x}^6\Big|\\
&\leq C\Big(\big(\sup_n\|\phi_n\|^5_{L^6_x}+\sup_n\|W^{M}_n\|^5_{L^6_x}\big)\|W^{M_1}_n\|_{L^6_x}+\|W^{M_1}_n\|^6_{L^6_x}\Big)\leq  \dfrac{\epsilon}{3}.\notag
\end{align}
Choosing $N_2\geq N_1$ such that $n\geq N_2$,  then (\ref{normsum}) yields
\begin{align}
\label{est2}
\Big|\|\phi_n-W^{M_1}_n\|_{L^6_x}^6-\sum^{M_1}_{j=1}\|e^{-it^j_n\Delta}\psi^j\|_{L^6_x}^6\Big|\leq \dfrac{\epsilon}{3}.
\end{align}
Since $W^{M}_n-W^{M_1}_n=\sum^{M_1}_{j=M+1}e^{-it^j_n\Delta}\psi^j(\cdot-x_n^j)$, by (\ref{normsum}), there exist $N_3\geq N_2$ such that $N_3\leq n$, 
\begin{align}
\label{est3}
\Big|\|W^{M}_n-W^{M_1}_n\|_{L^6_x}^6-\sum^{M_1}_{j=M+1}\|e^{-it^j_n\Delta}\psi^j\|_{L^6_x}^6\Big|\leq  \dfrac{\epsilon}{3}.
\end{align}
Thus, for $N_3\geq n$, (\ref{est1}), (\ref{est2}), and (\ref{est3}) yield
\begin{align}
\Big|\|\phi_n\|_{L^6_x}^6-\sum^{M}_{j=1}\|e^{-it^j_n\Delta}\psi^j\|_{L^6_x}^6-\|W^{M}_n\|_{L^6_x}^6\Big|\leq  \epsilon,
\end{align}
which concludes the proof.
\end{proof}


\begin{prop}[Nonlinear Profile decomposition]\label{nonradialPDNLS}
Let $\phi_n(x)$ be a uniformly bounded  sequence in $H^1(\Rn)$. Then for each $M\in\N$ there exists a subsequence of $\phi_n$, also denoted by  $\phi_n$,  for each $1 \leq j \leq M$, there exist a (same for all n) nonlinear profile $\tpsi^j$ in $H^1(\Rn)$, a sequence  of time shifts $t^j_n$, and a sequence  of space shifts $x^j_n$
and in addition, a sequence (in n) of remainders $\wt^M_n(x)$ in $H^1(\Rn)$,  such that 
\begin{equation}\label{NP}
\phi_n(x)=\sum^{M}_{j=1} \NLS(-t^j_n) \tpsi^j(x-x^j_n) + {\wt}^M_n(x), 
\end{equation}
where (as $n\to \infty$)
\begin{enumerate}
\item[(a)] for each j, either $t^j_n=0, t^j_n\to +\infty$ or $t^j_n\to -\infty$,
\item[(b)] if $t^j_n\to +\infty,$ then  $\|\NLS(-t)\tpsi^j\|_ {{S([0,\infty);\dHs)}}<+\infty$ 

and if $t^j_n\to -\infty$,  then
  $\|\NLS(-t)\tpsi^j\|_ {{S((-\infty,0];\dHs)}}<+\infty$,
\item[(c)] for $k \neq j$, then  $|t^j_n-t^k_n | + |x^j_n - x^k_n|\to+ \infty.$
\end{enumerate}

The remainder sequence has the following asymptotic smallness property:
\begin{align}
\label{smallnessPropNL}
\lim_{M \to \infty} \big( \lim_{n \to \infty} \| \NLS(t) \wt^M_n \|_  {\SHs}\big ) = 0.
\end{align}
For fixed $M\in\N$ and any $0 \leq s \leq 1$, we have the asymptotic Pythagorean expansion
\begin{align}
\label{HPdecompNL}
\| \phi_n \|^2_{\dH^{s}} = \sum^M_{j=1} \|\NLS(-t^j_n) \tpsi^j \|^2_{\dH^{s}} + \|\wt^M_n \|^2_{\dH^{s}} +o_n(1)
\end{align}
and the energy Pythagorean decomposition (note that $E[\NLS(-t^j_n)\tpsi^j]=E[\tpsi^j]$):
\begin{align}
\label{pythagoreanNLS}
E[\phi_n]=\sum_{j=1}^M E[\tpsi^j]+E[\wt^M_n]+o_n(1).
\end{align}
\end{prop}
\begin{proof}  From Proposition \ref{nonradialPD}, given that $\phi_n(x)$ is a uniformly bounded  sequence in $H^1$, we have
\begin{align}
\label{linearflowPD}
\phi_n(x)=\sum^{M}_{j=1} e^{-it^j_n\Delta} \psi^j(x-x^j_n) + W^M_n(x)
\end{align}
satisfying \eqref{time-space-seq}, \eqref{smallnessProp}, \eqref{limHs} and \eqref{pythagorean}. We will choose $M\in\N$ later. To prove this proposition, the idea is  to replace a linear flow $e^{it\Delta}\psi^j$ by some nonlinear flow.

Now for each $\psi^j$ we can apply the wave operator (Proposition \ref{Existence of wave operator}) to obtain a function $\tpsi^j\in H^1,$ which we will refer to as the nonlinear profile (corresponding to the linear profile $\psi^j$) such that the following properties hold:

\noindent For a given $j,$ there are two cases to consider: either $t^j_n$ is bounded, or $|t^j_n|\to +\infty.$
\medskip

\noindent\emph{ Case $|t^j_n|\to +\infty$:} 
If $t^j_n\to +\infty,$  Proposition \ref{Existence of wave operator} Part I \eqref{wave1} implies that

$
\|\NLS(-t^j_n)\tpsi^j-e^{-it^j_n\Delta}\psi^j\|_{H^1}\to 0 $   as  $t^j_n\to +\infty
$
and so
\begin{align}\label{cotacriticalnorm}
\|\NLS(-t)\tpsi^j\|_ {{S([0,+\infty),\dHs)}}<+\infty.
\end{align}
\noindent Similarly, if 
$t^j_n\to -\infty$, by \eqref{wave2} we obtain 
$
\|\NLS(-t^j_n)\tpsi^j-e^{-it^j_n\Delta}\psi^j\|_{H^1}\to 0$   as  $t^j_n\to -\infty,$
and hence,
\begin{align}\label{negcriticalnorm}
\|\NLS(-t)\tpsi^j\|_ {{S((-\infty,0],\dHs)}}<+\infty.
\end{align}

\noindent\emph{ Case  $t^j_n$ is bounded} (as $n\to\infty$): Adjusting the profiles $\psi^j$ we reduce it to the  case $t^j_n=0$. Thus, \eqref{time-space-seq}  becomes $|x^j_n-x^k_n|\to +\infty$ as $n\to\infty$, and continuity of the \emph{linear} flow in $H^1$, leads 
to $e^{-t_n^j\Delta}\psi^j \to \psi^j$ strongly in $H^1$ as $n \to \infty$. In this case, we simply let 
$\ds
\tilde{\psi}^j=\NLS(0)e^{-i(\lim_{n\to\infty}t_n^j)\Delta}\psi^j=e^{-i0\Delta}\psi^j=\psi^j. 
$

Thus, in either case of sequence $\{t_n^j\}$, we have a new nonlinear profile ${\tpsi^j}$ associated to each original linear profile $\psi^j$ such that 
\begin{equation}\label{flow approx}
\|\NLS(-t^j_n)\tpsi^j-e^{-it^j_n\Delta}\psi^j\|_{H^1}\to 0\quad\text{~~as~~}\quad n\to +\infty.
\end{equation}
Thus, we can substitute $e^{-it^j_n\Delta}\psi^j$ by $\NLS(-t^j_n){\tpsi^j}$ in \eqref{linearflowPD} to obtain
\begin{align}
\phi_n(x)=\sum^{M}_{j=1} \NLS(-t^j_n) \tpsi^j(x-x^j_n) + \wt^M_n(x),\label{profile+residuo}
\end{align}
where
\begin{align}
{\wt}^M_n(x)&={W}^M_n(x)+\sum_{j=1}^M\big\{
e^{-it^j_n\Delta}\psi^j(x-x^j_n)-\NLS(-t^j_n)\tpsi^j(x-x^j_n)\big\}\notag\\
&\equiv {W}^M_n(x)+\sum_{j=1}^M \at^j.
\label{residuo}
\end{align}
The triangle inequality yields 
$$
\|e^{it\Delta}{\wt}^M_n\|_  {\SHs}\leq\|e^{it\Delta}W^M_n\|_  {\SHs}+
c\sum_{j=1}^M\big\| e^{-it^j_n\Delta}\psi^j-\NLS(-t^j_n)\tpsi^j\big\|_  {\SHs}.
$$
By \eqref{flow approx} we have that
$
\|e^{it\Delta}{\wt}^M_n\|_  {\SHs}\leq\|e^{it\Delta}W^M_n\|_  {\SHs}+
c\sum_{j=1}^Mo_n(1),
$
and thus, 
$\ds
\lim_{M \to \infty} \big( \lim_{n \to \infty} \| e^{it\Delta} \wt^M_n \|_  {\SHs}\big ) = 0.$ 
Now we are going to apply a nonlinear flow to $\phi_n(x)$ and approximate it by a combination of ``nonlinear bumps" $\NLS(t-t^j_n)\tpsi^j(x-x^j_n),\;$ i.e., 
$
\NLS(t)\phi_n(x)\approx \sum^{M}_{j=1} \NLS(t-t^j_n) \tpsi^j(x-x^j_n).
$

Obviously, this can not hold for any bounded in $H^1$ sequence $\{\phi_n\}$, since, for example, a nonlinear flow can introduce finite time blowup solutions. However, under the proper conditions we can use  the long term perturbation theory (Proposition \ref{longperturbation}) to guarantee that a nonlinear flow behaves basically similar to the linear flow.

To simplify notation, introduce the nonlinear evolution of each separate initial condition $u_{n,0}=\phi_n$: $u_n(t,x)=\NLS(t)\phi_n(x),\;$
 the nonlinear evolution of each separate nonlinear profile (``bump"):
 $\;v^j(t,x)=\NLS(t)\tpsi^j(x), \;$
and  a linear sum of nonlinear evolutions of ``bumps":
$\tu_n(t,x)=\sum_{j=1}^M v^j(t-t^j_n,x-x^j_n).$

Intuitively, we think that $\phi_n=u_{n,0}$ is a sum of bumps $\tpsi^j$ (appropriately transformed) and $u_n(t)$ is a nonlinear evolution of their entire sum. On the other hand, $\tu_n(t)$ is a sum of nonlinear evolutions of each bump so we now want to compare $u_n(t)$ with $\tu_n(t)$. 

Note that if we had just the linear evolutions, then both $u_n(t)$ and $\tu_n(t)$ would be the same. 

Thus, $u_n(t)$ satisfies 
 $
i\partial_t u_n+\Delta u_n+| u_n|^{4} u_n=0,
$
and $\tu_n(t)$ satisfies 
 $
i\partial_t \tu_n+\Delta \tu_n+|\tu_n|^{4}\tu_n={\te^M_n},
$
where 
$
{\te^M_n}=|\tu_n|^{4}\tu_n-\sum_{j=1}^M |v^j_n(t-t^j_n,\cdot-x^j_n)|^{4} v^j_n(t-t^j_n,\cdot-x^j_n).
$
{\claim There exists a constant $A$ independent of $M$, and for every $M$, there exists $n_0=n_0(M)$ such that if $n>n_0$,  then
$
\|\tu_n\|_  {\SHs}\leq A.
$ \label{claim 1} 
}
{\claim  For each $M$ and $\epsilon>0$, there exists  $n_1=n_1(M, \epsilon)$ such that if $n>n_1$, then
$
\|{\te}^M_n\|_{L_t^{12/5}L_x^{6/5}}\leq \epsilon.
$ \label{claim 2}
}

\noindent We prove both claims at the end of this proof.
\medskip

Note $\tu_n(0,x)-u_n(0,x)=\wt_n^M(x)$. 
Then for any $\tilde \epsilon>0$ there exists $M_1=M_1(\tilde \epsilon)$ large enough such that for each $M>M_1$ there exists $n_2=n_2(M)$ with $n>n_2$ implying 
$$
\|e^{it\Delta}(\tu_n(0)-u_n(0))\|_  {\SHs}\leq\tilde \epsilon.
$$
Therefore, for $M$ large enough and $n=\max(n_0,n_1,n_2)$, since
 $$
 e^{it\Delta}(\tu_n(0))=e^{it\Delta}\Bigg(\sum_{j=1}^M v^j(-t^j_n,x-x^j_n)\Bigg),
 $$ 
 which are scattering by \eqref{flow approx},
Proposition \ref{longperturbation} implies 
$\|u_n\|_  {\SHs}< +\infty$, a  contradiction. 

Coming back to the nonlinear remainder $\wt^M_n,$ we estimate its nonlinear flow as follows (recall the notation of $\wt^M_n$, $\; W^M_n$ and $\at^j$ in \eqref{residuo}): 

By Strichartz estimates \eqref{stri} and by the triangle inequality, we get 
\begin{align}
\|\NLS(t)\wt^M_n\|_  {\SHs}&
\leq\|e^{it\Delta}\wt^M_n\|_  {\SHs}+\left\|\left|\wt^M_n\right|^{4}\wt^M_n\right\|_ {{\SdHs}}.\notag
\end{align}
And
\begin{align}\notag
\left\|\left|\wt^M_n\right|^{4}\wt^M_n\right\|_ {\SdHs}
\leq \left\|\left|D^{\frac12} \wt^M_n\right|^{4}\wt^M_n\right\|_ {L^{\frac{6}5}_t L^{\frac{3}2}_x}
\leq c\sum_{j=1}^M \|\at^j\|_ {L^{8}_t L^{8}_x}^{4} \|D^{\frac12}\at^j\|_ {L^{3}_t L^{6}_x}\\
\leq c\sum_{j=1}^M \|\at^j\|_  {\SHs}^{4} \|D^{\frac12}\at^j\|_ {S(\Lt)}
\leq c\sum_{j=1}^M \|\at^j\|_  {\SHs}^{4} \|\at^j\|_ {S(\dH^1)}\label{pb2}.
\end{align}
The $\SdHs$ norm is bounded by  $S'(\Lt)$ norm  which is estimated by $L^{\frac{3}{2}}_tL^{\frac{6}{5}}_x$ norm (the pair (3,6) is an $L^2$ admissible), apply Chain rule Lemma \ref{chain} followed by the  H\"older's inequality, and finally, the $L^{8}_tL^{8}_x$   and $L^{6}_tL^{3}_x$ norms are estimated by the $\SHs$ norm and $\SLt$ norm, respectively. And  $\;\dH^1\hookrightarrow\dHs,\;$  yields \eqref{pb2}.
Hence,
\begin{align}\label{evolution tw}
\|\NLS(t)\wt^M_n\|_  {\SHs}
\leq&\|e^{it\Delta}\wt^M_n\|_  {\SHs}\\\notag&+c\sum_{j=1}^M \big\|
e^{-it^j_n\Delta}\psi^j-\NLS(-t^j_n)\tpsi^j\big\|_{H^1}^{5}
\end{align}
and by \eqref{flow approx} the second term in \eqref{evolution tw} goes to zero as $n\to \infty$  and then applying \eqref{smallnessProp} the first term in \eqref{evolution tw} goes to zero as $M\to \infty,$ hence, we obtain 
$$
\lim_{n \to \infty} \|\NLS(t)\wt^M_n\|_  {\SHs}\to 0\quad \text{~~ as~~}\quad M\to \infty.
$$
Thus we proved  \eqref{smallnessPropNL} which completes the decomposition \eqref{NP}. This also gives \eqref{HPdecompNL}.

Next, we obtain the Energy Pythagorean decomposition. We substitute the linear flow in  Lemma \ref{energy Pythagorean expansion} by the nonlinear  
and repeat the above long term perturbation argument
to obtain
\begin{align}
\label{aproxL6NL}
\| \phi_n \|^{6}_{L^{6}} = \sum^M_{j=1} \|\NLS(-t^j_n) \psi^j \|^{6}_{L^{6}} + \|\wt^M_n \|^{6}_{L^{6}} +o_n(1),
\end{align} 
which yields the energy Pythagorean decomposition \eqref{pythagoreanNLS}. The proof will be concluded after we prove the Claims \ref{claim 1} and \ref{claim 2}.
\medskip

\noindent\emph{ Proof of Claim \ref{claim 1}}.
We show that for a large constant $A$ independent of $M$ and if $n>n_0=n_0(M)$, then 
$\|\tilde u_n\|_{S(\dHs)}\leq A.
$
 
 Let $M_0$ be a large enough such that
$\|e^{it\Delta}\wt^{M_0}_n\|_{\SHs}\leq \delta_{sd}.$   
Then, by \eqref{residuo}, for  each $j>M_0,$ we have $\|e^{it\Delta} \psi^j\|_{\SHs}\leq \delta_{sd},$ thus, Proposition \ref{Existence of wave operator} yields  
$
\|v^j\|_{\SHs}\leq 2\|e^{it\Delta} \psi^j\|_{\SHs}\; \mbox{    for  }\; j>M_0.
$

Recall the following inequality: for $a_j\geq 0,$ 
$\ds\Bigg|\bigg(\sum^M_{j=1}a_j\bigg)^4- \sum^M_{j=1}a_j^4\Bigg| \leq c_M \sum _{j\neq k}|a_j||a_k|^3.$
Then we have 
\begin{align}
\notag
\|\tilde u_n\|^{8}_{L_t^{8}L^{8}_x}&=\sum_{j=1}^{M_0}\|v^j\|^{8}_{L_t^{8}L^{8}_x}+\sum_{j=M_0+1}^{M}\|v^j\|^{8}_{L_t^{8}L^{8}_x} + \mbox{cross terms}\\
\label{sumsNLS}
&\leq\sum_{j=1}^{M_0}\|v^j\|^{8}_{L_t^{8}L^{8}_x}+2^{8}\sum_{j=M_0+1}^{M}\|e^{it\Delta}  \psi^j\|^{8}_{L_t^{8}L^{8}_x} + \mbox{cross terms,}
\end{align}
note that by (\ref{linearflowPD}) we have 
\begin{align}
\|e^{it\Delta}  \phi_n\|^{8}_{L_t^{8}L^{8}_x}=\sum_{j=1}^{M_0}\|e^{it\Delta}  \psi^j\|^{8}_{L_t^{8}L^{8}_x} +2^{8}\sum_{j=M_0+1}^{M}\|e^{it\Delta}  \psi^j\|^{8}_{L_t^{8}L^{8}_x} + \mbox{cross-terms.} \label{sumlinear}
\end{align}
Observe that by (\ref{time-space-seq}) and taking $n_0=n_0(M)$ large enough, we can consider $\{u_n\}_{n>n_0}$ and thus, make ``the cross terms" $\leq1$. 

Then (\ref{sumlinear}) and
  $\|e^{it\Delta}  \phi_n\|_{L_t^{8}L^{8}_x}\leq c \|\phi_n\|_{\dHs}\leq c_1$ imply $\sum_{j=M_0+1}^{M}\|e^{it\Delta}  \psi^j\|^{8}_{L_t^{8}L^{8}_x} $ is bounded independent of $M$ provided $n>n_0$. Thus, if $n>n_0$, (\ref{sumsNLS})  yields $\|\tilde u_n\|_{L_t^{8}L^{8}_x}$ is also bounded independent of $M$. 

 In a similar fashion, one can prove that $\|\tilde u_n\|_{L_t^{\infty}L^{4}_x}$ is bounded independent of $M$ provided $n>n_0$. Interpolation between these exponents gives $\|\tilde u_n\|_{L_t^{12}L^{6}_x}, $ which  is as well bounded independent of $M$ for $n>n_0$.
  To close the argument, we  apply Kato estimate (\ref{Kato-Strichartz}) to the integral equation of  
 $
i\partial_t \tilde u_n+\Delta \tilde u_n+|\tilde u_n|^4\tilde u_n={\te^M_n}. 
$
 Using   $\|\te_n^M\|_{S'(\dH^{-1/2})}\leq 1$ (Claim \ref{claim 2}), as in Proposition \ref{longperturbation}, we obtain that $\|\tilde u_n\|_{\SHs}$ is as well bounded independent of $M$ provided $n>n_0.$ Thus, Claim \ref{claim 1} is proved. 
\medskip

\noindent\emph{ Proof of Claim \ref{claim 2}}
 The expansion of ${\te^M_n}$ consist of $\sim M^5$ cross terms of the form
$\ds\prod_{k=1}^5v^{j_k}(t-t^{j_k}_n,x-x^{j_k}),$
where not all five $j_k$'s are the same. Without lost of generalization, assume that a pair  $j_1\neq j_2.$
We estimate simply by H\"older's
\begin{align*}
\bigl\|&\prod_{k=1}^5v^{j_k}(t-t^{j_k}_n,\cdot-x^{j_k})\bigr\|_{L_t^{12/5}L^{6/5}_x}
\\
&
\leq\| v^{j_1}(t-t^{j_1}_n,\cdot-x^{j_1}) v^{j_2}(t-t^{j_2}_n,\cdot-x^{j_2})\|_{L_t^{6}L^{3}_x}\prod_{m=3}^5 \|v^{j_m}(t-t^{j_m}_n,\cdot-x^{j_m})\|_{L_t^{12}L^{6}_x}.
\end{align*}
Note that either   $\{t_n^{j_1}\} \to \pm \infty$ or $\{t_n^{j_1}\}$ is bounded. 

If $\{t_n^{j_1}\} \to \pm \infty$, without loss of generalization assume $|t_n^{j_1}-t_n^{j_2}|\to\infty$ as $n\to \infty$ and by adjusting the profiles that $|x_n^{j_1}-x_n^{j_2}|\to0$  as $n\to \infty$. Since $v^{j_1}, v^{j_2}\in {L_t^{12}L^{6}_x}\hookrightarrow S(\dHs)$ 
$$\| v^{j_1}(t-(t^{j_1}_n-t^{j_2}_n),x) v^{j_2}(t,x)\|_{L_t^{6}L^{3}_x}\to 0.$$ 

If $\{t_n^{j_1}\}$ is bounded, without loss assume $|t_n^{j_1}-t_n^{j_2}|\to0$ and  $|x_n^{j_1}-x_n^{j_2}|\to\infty$ as $n\to \infty,$ then 
$\| v^{j_1}(t,x-(x^{j_1}-x^{j_2})) v^{j_2}(t,x)\|_{L_t^{6}L^{3}_x}\to 0,$
since $v^{j_1}, v^{j_2}\in {L_t^{12}L^{6}_x}\hookrightarrow S(\dHs)$.
Thus, in either case we obtain Claim \ref{claim 2}.
This finishes the proof of Proposition \ref{nonradialPDNLS}
\end{proof}

Observe that (\ref{HPdecompNL}) gives $\dH^1$ asymptotic orthogonality at $t=0$ and the following lemma extends it to the bounded NLS flow for $0\leq t\leq T.$ 

\begin{lemma}{\rm($\dH^1$ Pythagorean decomposition along the bounded NLS flow.)}\label{HPdecompNLf} Suppose $\phi_n$ is a bounded sequence in $H^1$. Let $T\in (0,\infty)$ be a fixed time. Consider the nonlinear profile decomposition from Proposition \ref{nonradialPDNLS}. Denote $\wt^M_n(t)\equiv \NLS(t)\wt^M_n$. Then for all $j$, the nonlinear profiles $v^j(t)\equiv \NLS(t)\tpsi^j$ exist up to time T and for all $t\in[0,T].$
\begin{align}
\label{PdecompNLS}
\|\nabla u_n(t)\|^2_{\Lt}=\sum_{j=1}^{M} \|\nabla v^j(t-t^j_n)\|^2_{\Lt}+\|\nabla \wt^M_n(t)\|^2_{\Lt_x}+o_n(1),
\end{align}
where $o_n(1)\to 0$ uniformly on $0\leq t\leq T.$
\end{lemma}
\begin{proof}
Let $M_0$ be such that for $M\geq M_0,$ we have $\| \NLS(t)\wt^M_n\|_{\SHs} \leq \delta_{sd}$ (as in Proposition \ref{small data}). Reorder the first $M_0$ profiles and denote by $M_2,$ $0\leq M_2\leq M,$  such that  
\begin{enumerate}
\item For each $1\leq j\leq M_2,$ we have $t^j_n=0.$ Observe that if $M_2=0,$  there are no $j$ in this case.
\item For each $M_2+1\leq j\leq M_0,$ we have $|t^j_n|\to \infty.$ If $M_2=0,$ then it means that  there are no $j$ in this case.
\end{enumerate}
From Proposition \ref{nonradialPDNLS} we have that $v^j(t)$ for $j>M_0$ are scattering and for a fixed T and $M_2+1\leq j\leq M_0$  we have $\|v^j(t-t^j_n)\|_{S(\dHs;[0,T])}\to 0$ as $n\to \infty$.  

In fact, taking $t_n^j \to +\infty$ and
$\|v^j(-t)\|_{S(\dot H^{1/2}; [0,+\infty))}<\infty$, dominated convergence leads  $\|v^j(-t)\|_{L_{[0,+\infty)}^q L_x^r}<\infty$, for $q<\infty$, and consequently,
$\|v^j(t-t_n^j)\|_{L_{[0,T]}^q L_x^r} \to 0$ as $n\to\infty$.  As   $v^j(t)$ has been constructed via the existence of wave operators to converge in $H^1$ to a linear flow at $\pm \infty$,  the $L^4_x$ decay of the linear flow together with the $H^1$ embedding  yields $\|v^j(t-t^j)\|_{L^\infty_{[0,T]}L^4_x}\to 0$.

Let $B=\max(1,\lim_{n}\|\nabla u_n(t)\|_{L^\infty_{[0,T]}\Lt_x})<\infty$. For each $1\leq j\leq M_2$, let $T^j\leq T$ be the maximal forward time  such that $\|\nabla v^j\|_{L^\infty_{[0,T^j]}\Lt_x}\leq 2B$. Denote by $\tilde T=\min_{1\leq j\leq M_2}T^j,$ or  $\tilde T=T$ if $M_2=0.$  It is sufficient to prove that (\ref{PdecompNLS}) holds for $\tilde T=T,$ since then for each $1\leq j \leq M_2$ we will have $T^j=T$, and therefore, $\tilde T=T.$ Thus, let's consider $[0,\tilde T]$. For each $1\leq j\leq M_2$,  we have
\begin{align}
\|v^j(t)\|&_{S(\dHs;[0,\tilde T])}
\lesssim \|v^j\|_{L^{\infty}_{[0,\tilde T]}L^4_x}
+\|v^j\|_{L^{4}_{[0,\tilde T]}L^{(4^+ )'}_x}\label{line1}\\
&\lesssim \|v^j\|^{1/2}_{L^{\infty}_{[0,\tilde T]}L^2_x}\|v^j\|^{1/2}_{L^{\infty}_{[0,\tilde T]}L^
{\infty}_x}
+ \|v^j\|^{1/2}_{L^{2}_{[0,\tilde T]}L^{\infty}_x}\|v^j\|^{1/2}_{L^{\infty}_{[0,\tilde T]}L^{(4^+ )'}_x}\label{line2}\\
&\lesssim \|v^j\|^{1/2}_{L^{\infty}_{[0,\tilde T]}L^2_x}\|\nabla v^j\|^{1/2}_{L^{\infty}_{[0,\tilde T]}L^
{2}_x}
+ \|v^j\|^{1/2}_{L^{2}_{[0,\tilde T]}L^{\infty}_x}\|v^j\|^{1/2}_{L^{\infty}_{[0,\tilde T]}\dH^{1-\frac2{(4^+ )'}}_x}\label{line2a}\\
&\lesssim \big(\|v^j\|^{1/2}_{L^{\infty}_{[0,\tilde T]}L^2_x}
+ \|v^j\|^{1/2}_{L^{2}_{[0,\tilde T]}L^{\infty}_x}\big)\|\nabla v^j\|^{1/2}_{L^{\infty}_{[0,\tilde T]}L^{2}_x}\label{line3}\\
&\lesssim \langle\tilde T^{1/2}\rangle B\label{line4},
\end{align}
note that  \eqref{line1} comes from  the ``end point" admissible  Strichartz norms ($L^{4}_{t}L^{(4^+ )'}_x$ and $L^{\infty}_{t}L^4_x$) since  all  other $S(\dHs)$ norms will be bounded by interpolation; \eqref{line2} is obtained using H\"older's inequality;  the Sobolev's embedding $\dH^1\hookrightarrow L^\infty$  and $\dH^{1-\frac{2}{(4^+ )'}}\hookrightarrow L^{(4^+ )'}$ leads  to \eqref{line2a}; since ${(4^+ )'}$ is large, we have the Sobolev's embedding $\dH^{1}\hookrightarrow \dH^{1-\frac{2}{(4^+ )'}}$ implying \eqref{line3}, and finally,
since  $\|v^j\|_{L^{\infty}_{[0,\tilde T]}L^{2}_x}=\|\psi^j\|_{L^{2}_x}\leq \lim_n \|\phi_n\|_{L^{2}}$ obtained from (\ref{HPdecompNL}) with $s=0$, we have \eqref{line4}.

As in proof of Proposition \ref{nonradialPDNLS}, set 
 $\tilde u_n(t,x)=\sum_{j=1}^M v^j(t-t^j_n,x-x^j_n)$
 and $ {\te^M_n}=i\partial_t \tilde u_n+\Delta \tilde u_n+|\tilde u_n|^4\tilde u_n.$ Thus,  for $M>M_0$ we have
\medskip

\noindent{\it Claim \ref{claim 1}.}
There exist a constant $A=A(\tilde T)$ independent of $M$, and for every $M$, there exists $n_0=n_0(M)$ such that if $n>n_0$,  then
$
\|\tilde u_n\|_{\SHs}\leq A.
$ 

\noindent{\it Claim  \ref{claim 2}.}
  For each $M$ and $\epsilon>0$, there exists  $n_1=n_1(M, \epsilon)$ such that for $n>n_1$, then
$
\|{\te}^M_n\|_{L_t^{12/5}L_x^{6/5}}\leq \epsilon.
$

\begin{rmk} \label{adapting long time}
Note  since $u(0)-\tilde u_n(0)=W^M_n$, there  exists $M'=M'(\epsilon)$ large enough so that for each $M>M'$  there exists  $n_2=n_2(M)$ such  that $n>n_2$ implies 
$$\|e^{it\Delta}(u(0)-\tilde u_n(0))\|_{S(\dHs;[0,\tilde T])}\leq \epsilon.$$
\end{rmk}
 Thus, the long time perturbation argument\footnote{Note that in Proposition \ref{longperturbation}, $T=+\infty$, while here, it is not necessary. However,  $T$ does not form part of the parameter
dependence since $\epsilon_0$ depends only on $A=A(T)$, not on $T$, that is, there will be dependence on $T$, but it is only through $A$} (Proposition \ref{longperturbation}) gives us $\epsilon_0=\epsilon_0(A).$
 Selecting an arbitrary $\epsilon\leq \epsilon_0,$ and from Remark \ref{adapting long time} take $M'=M'(\epsilon)$.  Now select an arbitrary $M>M'$ and take $n'=\max(n_0,n_1,n_2)$. Then combining Claims \ref{claim 1}, \ref{claim 2}, Remark \ref{adapting long time} and Proposition \ref{nonradialPDNLS}, we obtain that for $n>n'(M,\epsilon)$ with $c=c(A)=c(\tilde T)$ we have
 \begin{align}\label{aprox}
\|u_n-\tilde u_n\|_{S(H^{1/2};[0,\tilde T])}\leq c(\tilde T)\epsilon.
\end{align}

We will next prove \eqref{PdecompNLS} for $0\leq t\leq\tilde T$. Recall that $\|v^j(t-t^j_n)\|_{S(\dHs;[0,\tilde T])}\to 0$ as $n\to \infty$ and  for each $1\leq j\leq M_2$, we have $\|\nabla v^j\|_{L^\infty_{[0,T^j]}\Lt_x}\leq 2B$. By Strichartz estimates, 
$\|\nabla v^j(t-t^j_n)\|_{L^{\infty}_{[0,\tilde T]}\Lt_x}\lesssim \|\nabla v^j(-t^j_n)\|_{L^{\infty}_{[0,\tilde T]}\Lt_x}$, then
\begin{align*}
\|\nabla \tilde u_n(t)\|_{L^{\infty}_{[0,\tilde T]}\Lt_x}^2&=\sum_{j=1}^{M_2}\|\nabla v^j(t)\|_{L^{\infty}_{[0,\tilde T]}\Lt_x}^2+\sum_{j=M_2+1}^{M}\|\nabla v^j(t-t^j_n)\|_{L^{\infty}_{[0,\tilde T]}\Lt_x}^2+o_n(1)\\
&\lesssim M_2B^2+\sum_{j=M_2+1}^{M}\|\nabla \NLS(-t^j_n)\tpsi^j\|_{\Lt_x}^2+o(1)\\
&\lesssim M_2B^2+\|\nabla \phi_n\|_{\Lt_x}^2+o_n(1)
\lesssim M_2B^2+B^2+o_n(1).
\end{align*}
Using (\ref{aprox}), we obtain
\begin{align*}
\|u_n-\tilde u_n\|&_{L^\infty_{[0,\tilde T]}L^6_x}\lesssim \|u_n-\tilde u_n\|_{L^\infty_{[0,\tilde T]}L^4_x}^{2/3}\|u_n-\tilde u_n\|_{L^\infty_{[0,\tilde T]}L^{\infty}_x}^{1/3} \\
&\lesssim \|u_n-\tilde u_n\|^{1/6}_{L^{\infty}_{[0,\tilde T]}L^2_x}\|\nabla (u_n-\tilde u_n)\|^{1/6}_{L^{\infty}_{[0,\tilde T]}L^{2}_x}
\|\nabla(u_n-\tilde u_n)\|_{L^\infty_{[0,\tilde T]}\Lt_x}^{1/3} \\
&\lesssim c(\tilde T)^{1/6}(M_2B^2+B^2+o(1))^{1/6}\epsilon^{1/3}.
\end{align*}

Similar to the argument in the proof of (\ref{aproxL6NL}), we establish that for $0\leq t \leq\tilde T$ 
\begin{align}
\label{aproxuL6NL}
\| u_n(t) \|^6_{L^6} = \sum^M_{j=1} \|v^j(t-t^j_n) \|^6_{L^6} + \|\wt^M_n(t) \|^6_{L^6} +o_n(1).
\end{align} 

Energy conservation and (\ref{pythagoreanNLS}) give us
\begin{align}
E[u_n(t)]
=\sum_{j=1}^M E[\psi^j]+E[\wt^M_n]+o_n(1).
\label{energyNLF} \end{align}
Combining (\ref{aproxuL6NL}) and (\ref{energyNLF}) completes the proof.
\end{proof}

\section{Proofs of claims in Step 1 and Step 2 for scattering}\label{step 1and 2}

 \begin{prop}[Existence of a critical solution.]\label{existence of u_c}  There exists a global  
 solution $u_\crit(t)\in H^1(\Rn)$ with initial datum $u_{c,0}\in H^1(\Rn)$ such that 
$\|u_{\crit,0}\|_{\Lt}=1$, $
E[u_\crit]=(ME)_\crit<M[\uQ] E[\uQ], 
\g_{u_c}(t)<1$
 for all $0\leq t<+\infty,$
\begin{equation}\label{critical-besov}
\mbox{and}\quad \quad\|u_\crit\|_ {{\SHs}}=+\infty.
\end{equation}
\end{prop} 

\noindent Note that the condition $E[u_\crit]=(ME)_\crit<M[\uQ] E[\uQ] $ is equivalente to $\ME[u_c]<1.$
\begin{proof}
Consider a sequence of  solutions $u_n(t)$ to $\NLSf$ with corresponding initial data $u_{n,0}$ such that $\g_{u_n}(0)<1$
and $M[u_n] E[u_n] \searrow (ME)_\crit$ as $n\to +\infty$, for which $SC(u_{n,0})$ does not hold for any $n$. 

Without lost of generality, rescale the solutions so that $\|u_{n,0}\|_{\Lt}=1$, thus,
$$\|\nabla u_{n,0}\|_{\Lt}<\|\uQ\|_{\Lt}\|\nabla \uQ\|_{\Lt} \quad\mbox{ and }\quad E[u_n]\searrow (ME)_\crit.$$ 
By construction, $\|u_n\|_ {{\SHs}}=+\infty$. Note that  the sequence $\{u_{n,0}\}$ is uniformly bounded on $H^1$. Thus, applying the nonlinear profile decomposition (Proposition \ref{nonradialPDNLS}), we have
\begin{align}\label{profileinitial}
u_{n,0}(x)=\sum_{j=1}^M\NLS(-t^j_n)\tpsi^j(x-x_n^j)+\wt^M_n(x).
\end{align}
Now we will refine the profile decomposition property (b) in Proposition \ref{nonradialPDNLS} by using part II of Proposition \ref{Existence of wave operator} (wave operator), since it is specific to our particular setting here. 

Recall that in nonlinear profile decomposition we considered 2 cases when $|t_n^j|\to \infty$ and $|t_n^j|$ is bounded. In the first case, we can refine it to the following:

First note that we can obtain $\tpsi^j$ (from linear $\psi^j$) such that 
$$\|\NLS(-t^j_n)\tpsi^j-e^{-it^j_n\Delta}\psi^j\|_{H^1}\to 0\quad\text{~~as~~}\quad n\to +\infty$$ 
with properties \eqref{wave meg} and \eqref{wave scat}, since the linear profiles $\psi^j$'s satisfy 
$$\sum_{j=1}^{M}M[e^{-it^j\Delta}\psi^j]+\lim_{n\to+\infty}M[W^M_n]=\lim_{n\to+\infty}M[u_{n,0}]=1,$$
thus, $M[\psi^j]\leq1.$ Also, 
$$\sum_{j=1}^M\lim_{n\to+\infty} E[e^{-it^j_n\Delta}\psi^j] +\lim_{n\to+\infty}E[ W^M_n]=\lim_{n\to+\infty}E[u_{n,0}]=(ME)_\crit,$$
since each $E[e^{-it^j_n\Delta}\psi^j] \geq0$ (Lemma \ref{equivalence grad and energy}), we have 
$$\lim_{n\to\infty} E[e^{-it^j_n\Delta}\psi^j]\leq(ME)_\crit$$
and thus,
$$\frac12\|\psi^j\|^2_{\Lt}\|\nabla\psi^j\|^2_{\Lt}\leq M[\psi^j]\lim_{n\to\infty} E[e^{-it^j_n\Delta}\psi^j]\leq(ME)_\crit.$$

The properties \eqref{wave meg} for $\tpsi^j$ imply that  $\ME[\tpsi^j]<(ME)_\crit,$ and thus we get that 
\begin{equation}\label{nonlinear Besov}
\|\NLS(t)\tpsi^j(\cdot-x^j_n)\|_ {{\SHs}}<+\infty.
\end{equation}

This fact will be essential for the case 1 below. Otherwise, in the nonlinear decomposition \eqref{profileinitial} we also have the Pythagorean decomposition for mass and energy:
$$
\sum_{j=1}^M\lim_{n\to+\infty} E[\tpsi^j] +\lim_{n\to+\infty}E[ \wt^M_n]=\lim_{n\to+\infty}E[u_{n,0}]=(ME)_\crit,
$$
so we have \eqref{psi-identity} with $\sigma=\frac{1}{\sqrt 2}.$ Again, since each energy is greater than 0 (Lemma \ref{equivalence grad and energy}), for all $j$ we obtain
\begin{align}
\label{energy initial}
E[\tpsi^j]\leq(ME)_\crit.
\end{align}

Furthermore,  $s=0$ in (\ref{HPdecompNL}) imply
\begin{align}\label{masainitial}
\sum_{j=1}^{M}M[\tpsi^j]+\lim_{n\to+\infty}M[\wt^M_n]=\lim_{n\to+\infty}M[u_{n,0}]=1.
\end{align}

We show that  in the profile decomposition \eqref{profileinitial}  either 
more than one profiles $\tpsi^j$ are non-zero, or 
only one profile $\tpsi^j$ is non-zero  and the rest ($M-1$) profiles are zero. 
The first case will give a contradiction to the fact that each $u_n(t)$ does not scatter, consequently, only  the second possibility holds. That non-zero profile $\tpsi^j$ will be the initial data $u_{c,0}$
 and will produce the critical solution $u_\crit(t)=\NLS(t)u_{c,0},$ such that $\|u_\crit\|_ {{\SHs}}=+\infty.$
\medskip

\noindent\emph{ Case 1:} More than one $\tpsi^j\neq0.$  For each $j,$ \eqref{masainitial}  gives $M[\tpsi^j]<1$   and for a large enough $n$,  \eqref{energy initial} and \eqref{masainitial} yield
\begin{align*}
M[\NLS(t)\tpsi^j] E[\NLS(t)\tpsi^j]=M[\tpsi^j]  E[\tpsi^j]<(ME)_\crit.
 \end{align*}
 Recall \eqref{nonlinear Besov}, we have
 $$
\|\NLS(t-t^j)\tpsi^j(\cdot-x^j_n)\|_ {{\SHs}}<+\infty, \quad\quad \text{~ for large enough~} n,
$$
and thus, the right hand side in \eqref{profileinitial} is finite in $S(\dHs),\;$ since \eqref{smallnessPropNL} holds for the remainder $\wt_n^M(x).\;$ This contradicts the fact that $\|\NLS(t)u_{n,0}\|_ {{\SHs}}=+\infty$.

\noindent\emph{ Case 2:} Thus, we have that only one profile $\tpsi^j$ is non-zero, renamed to be $\tpsi^1$, 
\begin{equation}\label{initial approx}
u_{n,0}=\NLS(-t^1_n)\tpsi^1(\cdot-x^1_n)+\wt^1_n,
\end{equation} 
with 
$M[\tpsi^1]\leq 1,\;E[\tpsi^1]\leq(ME)_\crit\; \text{~and~} \;
\lim_{n \to+\infty} \|\NLS(t)\wt^1_n\|_ {{\SHs}}=0.
$

Let $u_\crit$ be the solution to $\NLSf$ with the initial condition $u_{c,0}=\tpsi^1$. Applying $\NLS(t)$ to both sides of \eqref{initial approx} and estimating it in $ {{\SHs}}$, we obtain (by the nonlinear profile decomposition Proposition \ref{nonradialPDNLS}) that
\begin{align*}
 \|u_\crit\|_ {{\SHs}} 
&= \|\NLS(t)\tpsi^1\|_ {{\SHs}}=\lim_{n\to\infty}
\|\NLS(t-t_n^1)\tpsi^1(\cdot-x^1_n)\|_ {{\SHs}}\\
&=\lim_{n\to\infty}\|\NLS(t)u_{n,0}\|_ {{\SHs}}=\lim_{n\to\infty}\|u_{n}(t)\|_ {{\SHs}}=+\infty,
\end{align*}
since by construction $\|u_n\|_ {{\SHs}}=+\infty,$ completing the proof.
\end{proof}

The proofs of the following Lemma \ref{precompact}, Lemma \ref{spacial translation} and  Proposition  \ref{precompact-localization} are very close to the ones in  \cite{HoRo08,DuHoRo08, HoRo09}, and thus, we omit them.

\begin{lemma}\label{precompact}{\rm( Precompactness of the flow of the critical solution.)} Assume $u_c$ as in Proposition \ref{existence of u_c}, there is a continuous path $x(t)$ in $\R^2$ such that  
$$K=\{u_c(\cdot-x(t),t)|t\in[0,+\infty)\}  \subset H^1$$
Then K is precompact in $H^1$.
\end{lemma}
\begin{lemma}
\label{spacial translation}
Let $u(t)$ be a solution of \eqref{quintic} defined on $[0,+\infty)$ such that $P[u]=0$ and either
\begin{itemize}
\item[a.] $K=\{u(\cdot-x(t),t) | t \in[0,+\infty)\}$ precompact in $H^1$, 
or
\item[b.] for all t,
\begin{align}
\|u(t)-e^{i\theta(t)}Q(\cdot-x(t))\|_{H^1}\leq \epsilon_1
\end{align}
\end{itemize}
for some continuous function $\theta(t)$ and $x(t).$ Then $\lim_{t\to+\infty}\frac{x(t)} {t}  = 0.$ 
\end{lemma}

\begin{cor}\label{precompact-localization}{\rm( Precompactness of the flow implies uniform localization.)} Assume $u$ is a solution to \eqref{quintic}  such that
$$K=\{u_c(\cdot-x(t),t)|t\in[0,+\infty)\} $$
is precompact in $H^1$. Then for each $\epsilon>0$, there exists $R>0$, so that for all $0\leq t<\infty$
$$\int_{|x+x(t)|>R}|\nabla u(x,t)|^2+|u(x,t)|^2+|u(x,t)|^6dx\leq\epsilon,$$
furthermore,
$\|u(t,\cdot-x(t))\|_{H^1(|x|\geq R)}\leq \epsilon.$
\end{cor}

\begin{thm}{\rm (Rigidity Theorem.)}\label{rigidity}
Let $u_0 \in H^1$   satisfy $P[u_0]=0$,  $\ME[u_0] <1$ and $\g_u(0)<1$. 
Let $u$ be the global $H^1$ solution of $\NLS^+_5(\R^2)$ with initial data $u_0$ and 
suppose that $K=\{u_c(\cdot-x(t),t)|t\in[0,+\infty)\} $ is precompact in $H^1.$ Then $u_0\equiv0$.
\end{thm}

\begin{proof}
Let $\phi \in C^{\infty}_0$ radial, such that
$\phi(x) =|x|^2$ for $|x| \leq 1$ and vanishing for $|x| \geq 2$.
For $R>0$ define
\begin{align}
z_R(t) = \int R^2 \phi(\frac{x}{R}) |u(x,t)|^2 dx.
\label{localization_z}
\end{align} 
Then direct calculations yield
$z'_R(t) =2\im \int R \nabla \phi(\frac{x}{R})\cdot \nabla u(t)\bar u(t) dx
$
 and H\"older's inequality leads to
\begin{align}
|z_R'(t)| \leq c R \int_{ \{ |x| \leq 2R \} } | \nabla u(t) | |u(t)| dx \leq c R \| u(t)\|_{\Lt} \| \nabla u(t) \|_{\Lt}
\label{localization_zder}.
\end{align}
Note that, 
\begin{align} \label{virialz}
z''_R(t)  = 4 \int \phi'' \bigg( \frac{|x|}{R} \bigg) | \nabla u |^2 - \frac{1}{R^2} \int \Delta^2 \phi \bigg ( \frac{|x|}{R} \bigg) |u|^2 -\frac{4}{3} \int \Delta \phi  \bigg( \frac{|x|}{R} \bigg) |u|^6.
\end{align}
Since $\phi$ is radial, we have
\begin{align}\label{ridigity-eq1}
z''_R(t)  = 8 \int | \nabla u |^2 - \frac{16}{3} \int |u|^6+A_R(u(t)),
\end{align}
where 
\begin{align*}
A_R(u(t)) = 4 \int \Bigg(\phi'' \bigg( \frac{|x|}{R} \bigg)-2\Bigg) | \nabla u |^2 +4 \int_{R\leq|x|\leq2R}\phi'' \bigg( \frac{|x|}{R} \bigg)| \nabla u |^2\\ - \frac{1}{R^2} \int \Delta^2 \phi \bigg ( \frac{|x|}{R} \bigg) |u|^2 
-\frac{4}{3} \int \Bigg(\Delta \phi  \bigg( \frac{|x|}{R} \bigg)-4\Bigg) |u|^6.
\end{align*}
Thus, 
\begin{align}
\label{no radial remainder}
\big|A_R(u(t))\big| \leq c \int_{|x|\geq R} \bigg(| \nabla u(t) |^2 + \frac{1}{R^2}  |u(t)|^2 + |u(t)|^6\bigg)dx.
\end{align}
Choosing $R$ large enough, over a suitably chosen time interval $[t_0,t_1]$, with $0\ll t_0\ll t_1<\infty$, it follows that
\begin{align}
\label{second derivative}
|z''(t)|\geq 16(1-\omega)E[u]-|A_R(u(t))|.
\end{align}
In Corollary \ref{precompact-localization} take $\epsilon=\frac{1-\omega}{c}$, with $c$ as in \eqref{no radial remainder}, we can take $R_0\geq0$ such that for all $t$,
\begin{align}
\label{integral}
\int_{|x+x(t)|>R_0}\big(|\nabla u(t)|^2+|u(t)|^2+|u(t)|^6\big)\leq\frac{1-\omega}{c}E[u].
\end{align}
Thus combining (\ref{no radial remainder}), (\ref{second derivative}) and (\ref{integral}), taking $R=R_0+\sup_{t_0\leq t\leq t_1}|x(t)|$ leads to the fact that for all $t_0\leq t\leq t_1$, 
 \begin{align}
\label{derivada}
|z''(t)|\geq 8(1-\omega)E[u].
\end{align}
Choosing $\gamma=(1-\omega)\frac{E[u]}{c\|Q\|_{\Lt}\|\nabla Q\|_{\Lt}}$ and by Lemma \ref{spacial translation}, there exists $t_0\geq0$ such that for all $t\geq t_0$, we have $|x(t)|\leq \gamma t.$ Taking $R=R_0+\gamma t_1,$ we have that (\ref{derivada}) holds for all $t\in[t_0,t_1]$, then integrating it over this interval, we obtain
\begin{align*}
|z'_R(t_1)-z'_R(t_0)| \leq 8(1-\omega) E(t)(t_1-t_0).
\end{align*}
Moreover, for all $t\in[t_0,t_1]$
$$
|z_R'(t)|\leq cR\|u(t)\|_{\Lt}\|\nabla u(t)\|_{\Lt}\leq c\|Q\|_{\Lt}\|\nabla Q\|_{\Lt}(R_0+\gamma t_1).
$$
Combining last two inequalities and letting $t_1\to+\infty$, yields $E[u]=0,$ which is a contradiction unless $ u(t)\equiv 0$.
\end{proof}

This finishes the first part of Theorem A* (global existence and scattering).
\section{Weak blowup via Concentration Compactness}\label {wbup}
In this section, we complete the proof of  Theorem A*, i.e.,  we show the  weak blow up part II (b). First, recall variational characterization of the ground state.

\subsection{Variational Characterization of the Ground State\label{sec:var GS}}

Propositon \ref{lion}  is a restatement of  Theorem I.2 from \cite{Lions84}. It is adjusted for our case from  Proposition 4.4 \cite{HoRo09}.
\begin{prop} \label{lion}
There exists a function $\epsilon(\rho)$ defined for small $\rho>0$ with $\lim_{\rho\to 0}\epsilon(\rho)=0$, such that for all $u\in H^1(\R^2)$ with
\begin{align*}
\big| \|u\|_{L^6}-\|Q\|_{L^6}\big|+\big| \|u\|_{L^2}-\|Q\|_{L^2}\big|+\big| \|\nabla u\|_{L^2}-\|\nabla Q\|_{L^2}\big|\leq\rho,
\end{align*}
there is $\theta_0\in \R$ and $x_0\in \R^2$ such that 
\begin{align}\label{sca_lion}
\|u-e^{i\theta_0}Q(\cdot-x_0)\|_{H^1}\leq\epsilon(\rho).
\end{align}
\end{prop}
This Proposition shows that if a solution $u(t,x)$ is close to $Q$ in mass and energy, then it is close to $Q$ in $H^1,$ the phase and shift in space. The Proposition \ref{newchar} is a variant of Proposition 4.1 \cite{HoRo09}, rephrased for our case.
\begin{prop} \label{newchar} 
There exists a function $\epsilon(\rho)$, such that $\epsilon(\rho)\to 0$ as $\rho\to 0$  satisfying the following:  Suppose there is $\lambda >0$ such that 
\begin{align}\label{bound1la}
\bigg|\ME[u]-(2\lambda^2-\lambda^4)\bigg|\leq \rho \lambda^4
\end{align}
and
\begin{align}\label{bound2la}
\big|\g_u(t)
-\lambda\big|\leq \rho\left\{\begin{array}{c}\lambda^3 \mbox{  if  } \lambda\leq1 \\\lambda \mbox{   if  } \lambda\geq1\end{array}\right. .
\end{align}
Then there exist $\theta_0 \in \R$ and $x_0\in \R^2$ such that if $\beta=\frac{M[u]}{M[Q]}$
\begin{align*}
\Big\|u(x)-e^{i\theta_0}\lambda Q\big(\lambda(\beta^{-1/2}x-x_0)\big)\Big\|_{L^2}\leq\beta^{1/4}\epsilon(\rho),
\quad\quad\mbox{
~~~and} 
\\
\Big\|\nabla\Big[u(x)-e^{i\theta_0}\lambda Q\big(\lambda(\beta^{-1/2}x-x_0)\big)\Big]\Big\|_{L^2}\leq\lambda\beta^{-1/4}\epsilon(\rho). 
\end{align*}
\end{prop}
The proof is similar to the one in  \cite{HoRo09} and we omit it.

\subsection{Induction Step 0: Near Boundary Behavior} 
In order to prove the weak ``blow up" we will employ the concentration compactness type argument. For establishing the divergence behavior and not scattering it was first developed in \cite{HoRo09}.

\begin{definition}\label{me_line}  Let $\lambda>0$.  The horizontal line for which $M[u]=M[Q]$ and $\frac{E[u]}{E[Q]}=2\lambda^2-\lambda^4$ is called the ``{\bf mass-energy}" line for $\lambda$ (See Figure \ref{fig2}).
\end{definition}

\noindent Note that we either have $0<\lambda<1$ or $\lambda>1.$ Here we consider $\lambda>1.$  
\begin{figure}[htbp]
\begin{flushleft}
\epsfxsize=13 cm \epsfysize=8 cm   
\epsffile{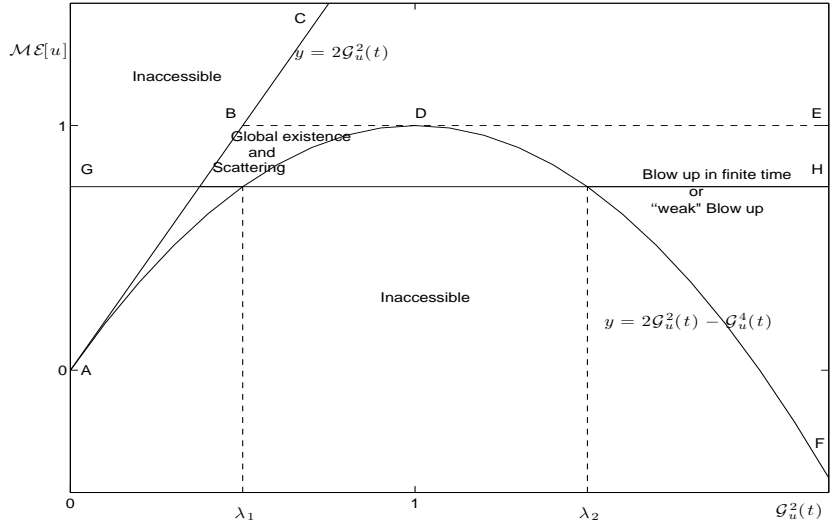}
\end{flushleft}
\caption{For a given $\lambda>0$ the horizontal line GH is referred as the ``mass-energy" line for this $\lambda$. Observe that this horizontal line can intersect the parabola $y=2\g_u^2-\g_u^4$ twice, i.e., it can be a ``mass energy" line for $0<\lambda_1<1$ and $1<\lambda_2<\infty$, the first case produces solutions which are global and are scattering (by Theorem A* part I) and the second case produces solutions which either blow up in finite time or diverge in infinite time (``weak" blow up) as shown in Section \ref{wbup}. }
 \label{fig2}
\end{figure}

 We will begin showing that the renormalized gradient $\g_u(t)$ cannot forever remain near the boundary if originally $\g_u(0)$ is very close to it. Next we would like to show that $\g_u(t)$ with initial condition $\g_u(0)>1$ close to the boundary on any ``mass-energy" line with $\ME[u]<1$ 
  will escape to infinity (along this line). To show this we assume to the contrary that for all solutions  (starting from some mass-energy line corresponding to initial renormalized gradient $\g_u(0)=\lambda_0>1$) are bounded in the renormalized gradient for all $t>0.$ And then conclude that this will lead to contradiction.

Theorem A*  part II (a) yields $\g_u(t)\geq1$ for all $t\in \R$ whenever $\g_u(0)\geq1$  on the ``mass-energy" line for some $\lambda>1$. Thus, a natural question is whether $\g_u(t)$  can be, with time, much larger than $\lambda$. We show (similar to \cite{HoRo09} Proposition 5.1) that it can not. 

\begin{prop}\label{near boundary case} Fix $\lambda_0>1.$ There exists $\rho_0=\rho_0(\lambda_0)>0$ (with the property that $\rho_0\to0$ as $\lambda_0\searrow1$), such that for any $\lambda\geq \lambda_0,$  there is NO solution $u(t)$ of $\NLS$ (\ref{quintic}) with P[u]=0 satisfying $\|u\|_{L^2}=\|Q\|_{L^2}$, and
$\frac{E[u]}{E[Q]}=2\lambda^2-\lambda^4$
(i.e., on any ``mass-energy" line corresponding to $\lambda\geq\lambda_0$)
with  $\lambda\leq\g_u(t)\leq\lambda(1+\rho_0)$ for all $t\geq0.$ 
A similar statement holds for $t\leq 0.$
\end{prop}
{\it Remark}: Note that this statement claims uniform ``non-closeness" to the boundary $DF$ (in Figure \ref{fig2}): if a solution lies on any ``mass-energy" line $\lambda$ (with $\lambda\geq \lambda_0$) and $\g_u(0)$ was close to the boundary $DF$, then eventually it will have to escape from this closeness, i.e.,  $\g_u(t^*)>\lambda(1+\rho_0)$ for some $t^*>0$.

\begin{proof}
To the contrary, assume that there exists a solution $u(t)$ of (\ref{quintic}) with $\|u\|_{L^2}=\|Q\|_{L^2}$,
$\frac{E[u]}{E[Q]}=2\lambda^2-\lambda^4$  for some $\lambda>\lambda_0$ and $\g_u(t)\in[\lambda, \lambda(1+\rho_0)]$.

By continuity of the flow $u(t)$ and Proposition \ref{newchar}, there are  continuous $x(t)$ and $\theta(t)$ such that 
\begin{align}\label{galieanQR}
\Big\|u(x)-e^{i\theta(t)}\lambda Q\big(\lambda(x-x(t)\big)\Big\|_{L^2}\leq\epsilon(\rho),
\end{align} 
\begin{align}\label{galieangradQR}
\Big\|\nabla\big[u(x)-e^{i\theta(t)}\lambda Q\big(\lambda(x-x(t)\big)\big]\Big\|_{L^2}\leq\lambda\epsilon(\rho). 
\end{align}

Define $R(T ) = \max\Big\{\max_{0\leq t\leq T} |x(t)|,\log \epsilon(\rho)^{-1}\Big\}$.  Consider the localized variance \eqref{localization_z}.  Note 
$$4\lambda^2E[Q]= \lambda^2{\|\nabla Q\|^2_{L^2}}\leq{\|\nabla u(t)\|^2_{L^2}},$$
and since $\frac{E[u]}{E[Q]}=2\lambda^2-\lambda^4$, we have
$$
z''_R= 32E[u]-8\|\nabla u\|_{L^2}^2+A_R (u(t))\leq -32 E[Q]\lambda^2(\lambda+1)(\lambda-1)+A_R (u(t)).
$$

Let $T > 0$ and  for the local virial identity (\ref{virialz}) assume $R = 2R(T )$. Therefore, (\ref{galieanQR}) and (\ref{galieangradQR}) assure that there exists $c_1 > 0$ such that 
$$|A_R (u(t))| \leq  c_1 \lambda^2\big( \epsilon(\rho) + e^{-R(T )}\big)^2 \leq \tilde c_1 \lambda^2 \epsilon(\rho)^2.$$ 
Taking a suitable $\rho_0$ small (i.e. $\lambda > 1$ is taken closer to 1), such that for $0\leq t\leq T$, $\epsilon(\rho)$ is small enough, we get 
$z''_ R (t) \leq -32 E[Q]\lambda^2(\lambda+1)(\lambda-1).$
Integrating $z''_R(t)$ in time over $[0, T]$ twice, we obtain 
$$\frac{z_R (T )} {T^ 2} \leq \frac{z_R(0)} {T^2} + \frac{z' _R (0)}{ T} - 16 E[Q]\lambda^2(\lambda+1)(\lambda-1).$$

Note $\sup_{x\in \R^2} \phi(x)$ from \eqref{localization_z}, is bounded, say by $c_2>0$. Then from \eqref{localization_z} we have 
$|z_R(0)|\leq c_2 R^2 \|u_0\|^2_ {L^2} =  c_2 R^2 \|Q\|^2_ {L^2},$
and by (\ref{localization_zder})
$$|z'_ R (0)| \leq c_3 R\|u_0\|_{L^2}\|\nabla u_0\|_{L^2} \leq  c_3 R\|Q\|_{L^2}\|\nabla Q\|_{L^2} \lambda(1 + \rho_0 ).$$ 
Taking T large enough so that by Lemma \ref{spacial translation} $\frac{R (T)}{ T}<\epsilon(\rho)$, we estimate
\begin{align*}
\frac{z_{2R(T)} (T )} {T^ 2} \leq& c_4 \Big(\frac{R(T)^2} {T^2} + \frac{R (T)}{ T}\Big) - 16E[Q]\lambda^2(\lambda+1)(\lambda-1)\\
\leq& C (\epsilon(\rho)^2+\epsilon(\rho)) - 16E[Q]\lambda^2(\lambda+1)(\lambda-1) .
\end{align*}
We can initially choose $\rho_0$ small enough (and thus, $\epsilon(\rho_0)$) such that 
$C (\epsilon(\rho)^2+\epsilon(\rho)) < 8E[Q]\lambda^2(\lambda+1)(\lambda-1).$ We obtain 
$0 \leq z_{2R(T )} (T ) < 0,$ which is a contradiction, showing that our initial assumption about the existence of a solution to (1.1) with bounded $\g_u(t)$ does not hold.
\end{proof}


Fix $\lambda>\lambda_0>1$. Consider a solution $u(t)$ of (\ref{quintic}) at the ``mass-energy" line for this $\lambda$. We showed that any such solution cannot have a renormalized gradient $\g_u(t)$ bounded near the boundary $DF$ for all time. We will show that  $\g_u(t)$, in fact, will tend to $+\infty$ (at least along an infinite time sequence). Again to the contrary assume that such solutions do have a uniform bound. 

We say  the property  $\GB(\lambda,\sigma)$ holds\footnote{$\GB$ stands for \emph{globally bounded gradient.}} if $\lambda\leq\g_u(t)\leq\sigma$  for all $ t\geq 0$, for some solutions on the ``mass-energy" line for $\lambda$.

  In other words, $\GB(\lambda, \sigma)$ is not true if for  every solution  $u(t)$ of (\ref{quintic}) at the ``mass-energy" line for $\lambda$ (for any $\lambda\leq\lambda_0>1$), such that $\lambda\leq\sigma<\g_u(t)$ for some $t>0$, then there exists $t^*$ such that $\sigma<\g_u(t^*)$. Iterating,
we conclude that, there exists a sequence $\{t_n\}\to \infty$ with $\g_u(t_n)>\sigma_n$ for all $n$ (and $\sigma_n\to+\infty$).  
  
Suppose $\GB(\lambda,\sigma)$ does not hold. Then for any $\sigma'<\sigma$ it does not hold either. This will allow us to induct on the $\GB$ notion.

\begin{definition}\label{thresholdGB}
Let $\lambda_0>1$.   We define {\rm the critical threshold} $\sigma_c$ by
$$
\sigma_c=  \sup\big\{\sigma | \sigma>\lambda_0 \mbox{ and } \GB(\lambda,\sigma) \mbox{ does NOT hold for all } \lambda \mbox{ with } \lambda_0\leq\lambda\leq\sigma\big\}.
$$  
Note that  $\sigma_c=\sigma_c(\lambda_0)$  stands for ``$\sigma$-critical".
\end{definition}

Notice that Proposition \ref{near boundary case} implies that $\GB(\lambda,\lambda(1+\rho_0(\lambda_0))$ does not hold for all $\lambda\geq\lambda_0$. 

\subsection{Induction argument}

Let $\lambda_0>1$ , we would like to show that $\sigma_c(\lambda_0)=+\infty$.  Let $u(t)$ be a solution to (\ref{quintic}) with initial condition $u_{n,0}$ such that
 $$M[u]=M[Q], \hspace{.4in} \frac{E[u]}{E[Q]}\leq2\lambda_0^2-\lambda^4_0\hspace{.4in} \mbox{and} \hspace{.4in}
\g_{u}(t)>1.$$
 
We want to show that there exists a sequence of times $\{t_n\}\to+\infty$ such that $\|\nabla u (t_n)\|_{L^2}\to\infty$. Assuming to the contrary, such sequence of times does not exist.  Let  $\lambda\geq\lambda_0$ be such that $\frac{E[u]}{E[Q]}=2\lambda^2-\lambda^4,$ and thus,  there exists $\sigma<\infty$ such that  $\lambda\leq \g_{u}(t)
\leq \sigma$ for all $t\geq0$, i.e., $\GB(\lambda,\sigma)$ holds with  $\sigma_c(\lambda_0)\leq\sigma<\infty.$ 
 
Now, we take $u(t)=u_c(t) $ to be the critical threshold solution given by Lemma \ref{existence threshold} (see below). Then by Lemma \ref{compactH1} we have uniform concentration of $u_c(t)$ in time, which together with the localization property (Lemma \ref{blowup localization} )  implies that $u_c(t)$ blows up in finite time, which contradicts the fact that $u_c(t)$ is bounded  in $H^1.$ As a result $u_c(t)$ cannot exist and this ends the proof of Theorem A*.

Before proceeding with the Existence Theorem we  introduce the profile reordering (Lemma \ref{reordering}) which together with the nonlinear profile decomposition of the sequence $\{u_{n,0}\}$  will allow us to construct  a ``\emph{critical threshold solution}" (see Existence of Threshold solution  Lemma \ref{existence threshold}).

\begin{lemma}{\rm(Profile reordering.)}\label{reordering} Suppose $\phi_n=\phi_n(x)$ is a bounded sequence in $H^1(\Rn)$.  
Assume that $M[\phi_n]=M[Q],$ $\frac{E[\phi_n]}{E[Q]}=2\lambda_n^2-\lambda^4_n$, such that  $1<\lambda_0\leq\lambda_n$ and $\lambda_n\leq \g_{\phi_n}(t)$
for each n. Apply Proposition \ref{nonradialPDNLS} to sequence $\{\phi_n\}$ and obtain nonlinear profiles $\{\tpsi^j\}$ Then, these profiles $\tpsi^j$ can be reordered so that there exist $1\leq M_1\leq M_2\leq M$ and 
\begin{enumerate}
\item For each $1\leq j \leq M_1,$  we have $t^j_n=0$ and $v^j(t)\equiv \NLS(t)\tpsi^j$ does not scatter as $t\to +\infty.$ (In particular, there is at least one $j$ in this case.)
\item For each $M_1+1\leq j\leq M_2,$  we  have $t^j_n=0$ and $v^j(t)$ scatters as $t\to +\infty.$ (If $M_1=M_2$, there are no $j$ with this property.)
\item For each $M_2+1\leq j\leq M,$  we  have $|t^j_n|\to\infty$ and $v^j(t)$ scatters as $t\to +\infty.$ (If $M_2=M$, there are no $j$ with this property.) 
\end{enumerate}
\begin{proof}
Pohozhaev identities and energy definition yield
\begin{align*}
\frac{\|\phi_n\|^6_{L^6}}{\|Q\|^6_{L^6}}=2\g^2_{\phi_n}(t)-\frac{E[\phi_n]}{E[Q]}
\geq\lambda^4_n\geq \lambda_0^4>1.
\end{align*}
Notice that if $j$ is such that $|t^j_n|\to \infty,$ then $L^6$  scattering yields  $\|\NLS(-t_n^j)\tpsi^j\|_{L^6}\to0,$ and by  (\ref{aproxL6NL}) we have that $\frac{\|\phi_n\|^6_{L^6}}{\|Q\|^6_{L^6}} \to 0$. Therefore, there exist at least one $j$ such that $t^j_n$ converges as $n\to \infty$. Without loss of generality, assume that $t^j_n=0,$ and reorder the profiles such that for $1\leq j\leq M_2$, we have $t^j_n=0$ and for $M_2+1\leq j\leq M$, we have $|t^j_n|\to\infty$.

It is left to prove that there is at least one $j$, $1\leq j\leq M_2$ such that $v^j(t)$ is not scattering. Assume then  for all $1\leq j\leq M_2$ we have that all $v^j$ are scattering, and thus, $ \|v^j(t)\|_{L^6} \to 0$ as ${t\to+\infty}.$ Let  $\epsilon > 0$ and $t_0$ large enough such that for all $1\leq j\leq M_2$ we have $ \|v^j(t)\|_{L^6}^6\leq \epsilon/M_2.$ Using the $L^6$ orthogonality (\ref{aproxuL6NL}) along the NLS flow, and letting $n\to +\infty$, we obtain
\begin{align*}
\lambda_0^4\|Q\|_{L^6}^6\leq& \| u_n(t) \|^6_{L^6} \\
&= \sum^{M_2}_{j=1} \|v^j(t_0) \|^6_{L^6} +\sum^{M}_{j=M_2+1} \|v^j(t_0-t^j_n) \|^6_{L^6} + \|W^M_n(t) \|^6_{L^6} +o_n(1)\\
&\leq \epsilon+ \|W^M_n(t) \|^6_{L^6} +o_n(1).
\end{align*}
The last line  is obtained, since $\sum^{M}_{j=M_2+1} \|v^j(t_0-t^j_n) \|^6_{L^6}\to 0$ as $n\to \infty$. This gives a contradiction.
\end{proof}
\end{lemma}

\begin{lemma}{\rm(Existence of the threshold solution.)} \label{existence threshold}
There exists initial data $u_{c,0}$ with $M[u_c]=M[Q]$ and $\lambda_0\leq\lambda_c\leq\sigma_c(\lambda_0)$ such that $u_c(t)\equiv\NLS(t)u_{c,0}$ is global,  $\frac{E[u_c]}{E[Q]}=2\lambda_c^2-\lambda_c^4$ and, moreover, $\lambda_c\leq\g_{u_c}(t)
\leq \sigma_c$ for all $t\geq0.$
\end{lemma}

\begin{proof}
Definition of $\sigma_c$ implies the existence of sequences $\{\lambda_n\}$ and $\{\sigma_n\}$ with $\lambda_0\leq\lambda_n\leq \sigma_n$  and $\sigma_n \searrow\sigma_c$ such that $\GB(\lambda_n,\sigma_n)$ is false.  This means that there exists $u_{n,0}$ with $M[u]=M[ \uQ]$, $\frac{E[u_{n,0}]}{E[ \uQ]}=2\lambda^2-\lambda^4$ and $\lambda_c\leq\frac{\|\nabla u\|_{L^2}}{\|\nabla  \uQ\|_{L^2}}=\g_u(t)\leq \sigma_c,$
 such that $u_n(t)=\NLS(t)u_{n,0}$ is global.

Note that the sequence $\{\lambda_n\}$ is bounded, thus we pass to a convergent subsequence $\{\lambda_{n_k}\}$. Assume $\lambda_{n_k}\to \lambda'$  as ${n_k}\to\infty$,  thus $\lambda_0\leq\lambda'\leq\sigma_c.$

We apply the nonlinear profile decomposition and reordering. 
In Lemma \ref{reordering}, let $\phi_n=u_{n,0}$. Recall that $v^j(t)$ scatters as $t\to \infty$ for $M_1+1\leq j\leq M_2,$  and by Proposition \ref{nonradialPDNLS}, $v^j(t)$ also scatter in one or the other time direction  for $M_2+1\leq j\leq M$ and  $E[\tpsi^j]=E[v^j]\geq0.$ Thus, by the Pythagorean decomposition for the nonlinear flow (\ref{pythagoreanNLS}) we have
$\sum_{j=1}^{M_1}E[\tpsi^j]\leq E[\phi_n]+o_n(1).$
For at least one $1\leq j\leq M_1$, we have $E[\tpsi^j]\leq\max\{\lim_n E[\phi_n],0\}$. Without loss of generality, we may assume $j=1$. Since $1=M[\tpsi^1]\leq\lim_nM[\phi_n]=M[ \uQ]=1$, it follows
$
\ME[\tpsi^1]\leq \max\bigg(\lim_n\dfrac{E[\phi_n]}{E[ \uQ]}\bigg),
$
thus, for some $\lambda_1\geq\lambda_0,$ we have
$
\ME[\tpsi^1]=2\lambda_1^2-\lambda_1^4.
$

Recall $\tpsi^1$ is a nonscattering solution, thus $\g_{\psi^1}(t)>\lambda$, otherwise it will contradict Theorem A* Part I (b). We have two cases:  either $\lambda_1\leq\sigma_c$ or $\lambda_1>\sigma_c$.

\noindent
\emph{ Case 1. } $\lambda_1\leq\sigma_c.$ Since the statement ``$\GB(\lambda_1,\sigma_c-\delta)$ is false" implies for each $\delta>0$, there is a nondecreasing sequence $t_k$ of times such that 
$
\lim[\g_{v^1}(t_k)]^{\frac1s}\geq \sigma_c,
$ 
thus,
\begin{eqnarray}
\sigma_c^2-o_k(1)&\leq&\lim[\g_{v^1}(t_k)]^{\frac2s}
\leq\dfrac{\|\nabla v^1(t_k)\|^2_{L^2}}{\|\nabla  \uQ\|^2_{L^2}}\notag\\
&\leq&\dfrac{\sum_{j=1}^M\|\nabla v^1(t_k-t_n)\|^2_{L^2}+\|W_n^M(t_k)\|^2_{L^2}}{\|\nabla  \uQ\|^2_{L^2}}\label{desigu}\\
&\leq&\dfrac{\|\nabla u_n(t)\|^2_{L^2}}{\|\nabla  \uQ\|^2_{L^2}}+o_n(1)\leq\sigma_c^2+o_n(1)\notag.
\end{eqnarray}
Taking $k\to \infty$, we obtain $\sigma_c^2-o_n(1)=\sigma_c^2+o_k(1)$. Thus, $\|W_n^M(t_k)\|_{H^1}\to 0$  and $M[v^1]=M[ \uQ].$ Then, Lemma \ref{HPdecompNLf} yields that for all $t,$
$$
\dfrac{\|\nabla v^1(t)\|^2_{L^2}}{\|\nabla  \uQ\|^2_{L^2}}\leq\lim_n\dfrac{\| u_n(t)\|^2_{L^2}}{\|\nabla  \uQ\|^2_{L^2}}\leq \sigma_c.
$$
Take $u_{c,0}=v^1(0)(=\psi^1),$ and $\lambda_c=\lambda_1.$

\noindent \emph{ Case 2. } $\lambda_1\geq\sigma_c.$ Note that 
$\lambda_1^2\leq\lim[\g_{v^1}(t_k)]^{\frac2s}.$ Thus,  replacing  \eqref{desigu} with this condition, taking $t_k=0$ and sending $n\to+\infty$, we obtain $\lambda_1\leq\sigma_c,$ which is a contradiction. Thus, this case cannot happen.
 \end{proof}

Let's assume $u(t)=u_c(t)$ to be the critical solution provided by Lemma \ref{existence threshold}.

\begin{lemma}{\label{compactH1}}
There exists a path $x(t)$ in $\R^2$ such that 
$$
K=\{u(\cdot-x(t),t)|t\geq 0\}\subset H^1
$$
has a compact closure in $H^1$.
\end{lemma}
The proof of this Lemma follows closely to the proof of  Lemma  9.1  in \cite{HoRo09} and we omit them.

\begin{lemma} [Blow up for \emph{a priori} localized solutions]\label{blowup localization} 
Suppose $u$ is a solution of the $\NLSf$ at the mass-energy line $\lambda>1$, with $\g_u(0)>1$. Select $\kappa$ such that
$0<\kappa<\min(\lambda-1,\kappa_0)$, where $\kappa_0$ is an absolute
constant. Assume that there is a radius $R\gtrsim \kappa^{-1/2}$
such that for all $t$, we have a localized gradient
$$\g_{u_{R}}(t):=\dfrac{\|u\|_{\Lt(|x|\geq R)}\|\nabla u(t)\|_{\Lt(|x|\geq R)}}{\|\uQ\|_{\Lt(|x|\geq R)}\|\nabla \uQ\|_{\Lt(|x|\geq R)}}\lesssim \kappa.$$
Define $\tilde r(t)$ to be the scaled local variance:
$
V_ R(t) =  \frac{z_R(t)}{32 E[\uQ] \left(\lambda^{2}(\lambda^2-1-\kappa)
\right)} \,,
$
where $z_R(t)$ is from \eqref{localization_z}. Then a blow up occurs in forward time before $t_b$ (i.e., $T^* \leq
t_b$), where
$
t_b = V_ R'(0) + \sqrt{ V_ R'(0)^2 + 2 V_ R(0)} \,.
$
\end{lemma}
\begin{proof}
By the local virial identity \eqref{ridigity-eq1},
$$
V_R''(t) =
\frac{32E[u]-8\|\nabla u\|_{L^2}^2+{A_R(u(t))}}
{32 E[\uQ] \left(\lambda^{2}(\lambda^2-1-\kappa)
\right)}
$$
where  
$
\big|A_R(u(t))\big|  = \| \nabla u(t) \|^2_{\Lt(|x|\geq R)}  + \frac{1}{R^2}  \|u(t)\|^2_{\Lt(|x|\geq R)} + \|u(t)\|^{6}_{L^{6}(|x|\geq R)}.
$

Note that, $4E[\uQ]=\|\nabla \uQ\|_{L^2}^2$ and definition of the mass-energy line yield

$$
\frac{32E[u]-8\|\nabla u\|_{L^2}^2}
{32 E[\uQ] }
=\frac{E[u]}{E[\uQ]}-\frac{\|\nabla u\|_{L^2}^2}{\|\nabla Q\|_{L^2}^2}=2\lambda^2
-\lambda^4-[\g_u(t)]^2$$
In addition, we have the following estimates 
$$\| \nabla u(t) \|^2_{\Lt(|x|\geq R)}\lesssim \kappa,\quad\quad
\frac{\| u(t) \|^2_{\Lt(|x|\geq R)}}
{R^2}=\frac{\| \uQ \|^2_{\Lt}}
{R^2}\lesssim \kappa,$$
\begin{align} \label{magia3}
 \|u(t)\|^{6}_{L^{6}(|x|\geq R)}&\lesssim
 \|\nabla u\|^{4}_{L^{2}(|x|\geq R)}\| u\|^{2}_{L^{2}(|x|\geq R)}\lesssim[\g_{u_{R}}(t)]^{2}
\left(\|\nabla \uQ\|^{2}_{L^{2}}\| \uQ\|^{2}_{L^2}\right)\lesssim \kappa.
\end{align}
We used the Gagliardo-Nirenberg to obtain \eqref{magia3} and noticing that $\|\nabla \uQ\|^{2}_{L^{2}}$ and $\| \uQ\|^{2}_{L^2}$ we estimated by $\kappa$ up to a constant. In addition, $\g_u(t)>1,$ then $\kappa\lesssim \kappa[\g_u(t)]^2$. Applying the above estimates, it follows
$$
V_ R''(t) \lesssim
\frac{2\lambda^2
-\lambda^4-[\g_u(t)]^2(1-\kappa)}
{\lambda^{2}(\lambda^2-1-\kappa)},$$ 
since $\g_u(t)\geq\lambda$, we obtain
$
V_ R''(t) \lesssim\frac{
\lambda^{2}(1+\kappa-\lambda^2)}
{\lambda^{2}(\lambda^2-1-\kappa)} \leq -1  \,,
$
which is a contradiction.  Now integrating in time twice gives
$
V_ R(t) \leq -\frac12 t^2 + V_ R'(0)t + V_ R(0) \,.
$
The positive root of the polynomial on the right-hand side is $t_b=V_ R'(0) + \sqrt{ V_ R'(0)^2 + 2 V_ R(0)}.$
\end{proof}

This finally finishes the proof of Theorem A*. 

Note that Theorem A can be extended to other nonlinearities and dimensions except that one needs to deal carefully with fractional powers, Strichartz estimates and others implications from that. We address it elsewhere \cite{guevara}.

\bibliographystyle{amsalpha}

\end{document}